\documentclass[a4paper, twoside, 11pt]{amsart}

\usepackage[margin=3cm]{geometry}                 
\usepackage{graphicx}
\usepackage{amssymb}
\usepackage{mathrsfs}
\usepackage{enumitem}
\usepackage{array,multirow}
\usepackage{mathtools}
\usepackage[usenames,dvipsnames]{xcolor}
\usepackage{dsfont}
\usepackage{calligra}
\DeclareMathAlphabet{\mathcalligra}{T1}{calligra}{m}{n}
\usepackage{tikz}
\usepackage{tikz-cd}
\usetikzlibrary{shapes.misc}
\usetikzlibrary{decorations.markings}
\usepackage{etoolbox}
\usepackage{float}
\usepackage{multicol}
\usepackage{rotating}
\mathtoolsset{showonlyrefs=true}
\usepackage{arydshln}
\usetikzlibrary{positioning}
\usepackage{caption}
\usepackage{subcaption}
\usepackage{mathdots}

\usepackage[]{hyperref}
\hypersetup{
    colorlinks=true,
    citecolor=blue,
    linkcolor=blue,
    filecolor=magenta,      
    urlcolor=cyan,
    }

\usepackage{soul}
\usepackage[textwidth=1.8cm]{todonotes}  
\usepackage{marginnote}
\usepackage{zref-savepos}

\newcount\todocount

\newcommand{\checkxpos}[3][]{%
  \ifdim \zposx{#2}sp < 20000000sp%
    \mynote[#1]{#3}%
  \else%
    \note[#1]{#3}%
  \fi%
}

\usepackage[T2A]{fontenc}
\usepackage[utf8]{inputenc}
\usepackage[russian, main=english]{babel}

\newcommand{\R}{\mathbb R}

\newcommand{\CC}[1]{\mathbb{C}^{#1}}

\newcommand{\RR}[1]{\mathbb{R}^{#1}}
\newcommand{\dw}{\frac{\partial}{\partial w}}
\newcommand{\z}{\zeta}
\newcommand{\dz}[1]{\frac{\partial}{\partial z_{#1}}}
\newcommand{\dzz}{\frac{\partial}{\partial \zeta}}

\newcommand{\re}{\ensuremath{\mbox{\rm Re}\,}}
\newcommand{\hol}{\mathfrak{hol}}
\newcommand{\aut}{\mathfrak{aut}}
\newcommand{\g}{\mathfrak g}

\newtheorem{theorem}{Theorem}[section]
\newtheorem{lemma}[theorem]{Lemma}

\newtheorem{proposition}[theorem]{Proposition}
\newtheorem{corollary}[theorem]{Corollary}

\newtheorem{remark}[theorem]{Remark}
\newtheorem{definition}[theorem]{Definition}

\makeatletter
\def\@tocline#1#2#3#4#5#6#7{\relax
  \ifnum #1>\c@tocdepth 
  \else
    \par \addpenalty\@secpenalty\addvspace{#2}%
    \begingroup \hyphenpenalty\@M
    \@ifempty{#4}{%
      \@tempdima\csname r@tocindent\number#1\endcsname\relax
    }{%
      \@tempdima#4\relax
    }%
    \parindent\z@ \leftskip#3\relax \advance\leftskip\@tempdima\relax
    \rightskip\@pnumwidth plus4em \parfillskip-\@pnumwidth
    #5\leavevmode\hskip-\@tempdima
      \ifcase #1
       \or\or \hskip 1em \or \hskip 2em \else \hskip 3em \fi%
      #6\nobreak\relax
    \dotfill\hbox to\@pnumwidth{\@tocpagenum{#7}}\par
    \nobreak
    \endgroup
  \fi}
\makeatother

\title[New examples of $2$-nondegenerate hypersurfaces]{New examples of $2$-nondegenerate real hypersurfaces in $\mathbb{C}^N$ with arbitrary nilpotent symbols}

\date{}

\author{Martin Kol\'{a}\v{r}}
\address{Martin Kol\'{a}\v{r},
	Department of Mathematics and Statistics, 
	Masaryk University,
	Kotl\'{a}\v{r}sk\'{a}~2,
	611 37 Brno,
	Czech Republic}\email{ mkolar@math.muni.cz}
\urladdr{\url{http://www.math.muni.cz/~mkolar}}

\author{Ilya Kossovskiy}
\address{Ilya Kossovskiy,
	Department of Mathematics and Statistics, 
	Masaryk University,
	Kotl\'{a}\v{r}sk\'{a}~2,
	611 37 Brno, Czech Republic,
 $\&$
 Institute of Discrete Mathematics and Geometry,
 TU Vienna,
	Austria}\email{ Kossovskiyi@math.muni.cz} 
\urladdr{\url{http://www.math.muni.cz/~kossovskiyi}}

\author{David Sykes}
\address{David Sykes,
	Department of Mathematics and Statistics, 
	Masaryk University,
	Kotl\'{a}\v{r}sk\'{a}~2,
	611 37 Brno,
	Czech Republic, $\&$ Institute of Discrete Mathematics and Geometry,
 TU Vienna,
	Austria}\email{ sykes@math.muni.cz}
\urladdr{\url{http://www.math.muni.cz/~sykes}}

\makeatletter
\@namedef{subjclassname@2020}{\textup{2022} Mathematics Subject Classification}
\makeatother
\thanks{The first and the second author were supported by the GACR grant GC22-15012J, the third author was supported by Austrian Science Fund (FWF): P34369.}
\begin{document}

\begin{abstract}
    We introduce a class of uniformly $2$-nondegenerate CR hypersurfaces in $\mathbb{C}^N$, for $N>3$, having a rank $1$ Levi kernel. The class is first of all remarkable by the fact that for every $N>3$ it forms an {\em explicit} infinite-dimensional family of everywhere $2$-nondegenerate hypersurfaces. To the best of our knowledge, this is the first such construction. Besides, the class an infinite-dimensional family of non-equivalent structures having a given constant nilpotent CR symbol for every such symbol. Using methods that are able to handle all cases with $N>5$ simultaneously, we solve the equivalence problem for the considered structures whose symbol is represented by a single Jordan block, classify their algebras of infinitesimal symmetries, and classify the locally homogeneous structures among them. We show that the remaining considered structures, which have symbols represented by a direct sum of Jordan blocks, can be constructed from the single block structures through simple linking and extension processes.
\end{abstract}

\maketitle
\tableofcontents

\section{Introduction}

 The main goal of this paper is further understanding the geometry of an intriguing and not yet well understood class of real {submanifolds} in complex space $\mathbb{C}^N$ called $2$-nondegenerate hypersufaces. The latter have been a subject of intensive research over recent decades, since \cite{freeman1974local,freeman1977local} and the inception of the $l$-nondegeneracy notion in \cite{bhr}. We refer to 
  e.g. \cite{ebenfeltC3,ebenfeltduke,kaup2,fels2008classification,iz,kaupzaitsev,ms,Poc} (see \cite{MP} for the published version of \cite{Poc}), and the introduction in \cite{kolar2019complete} for a historic outline in the lowest dimension $N=3$ where $2$-nondegeneracy occurs, and emphasize that in higher dimensions $N>3$ the study of $2$-nondegenerate hypersurfaces exhibits considerable further difficulties (compared to the $N=3$ case). Principally, in higher dimensions the structures have nontrivial CR symbols, a basic local invariant. We refer to \cite{porter2021absolute} where these symbols were introduced for some recent advances here, and mention that a natural question which remains unsolved within these developments is a description of the moduli space of the $2$-nondegenerate structures whose CR symbol is constant. In particular, little is known about the extent to which constancy of the CR symbol imposes rigidity phenomena, that is, whether for a given symbol the constant symbol structures are locally unique, exiguous, or existing in rich variety. Addressing the latter question for a large class of symbols is one of the main outcomes of the current paper.

{We introduce} new examples (Theorems \ref{main theorem} and \ref{second main theorem}) of holomorphically nondegenerate hypersurfaces in $\mathbb{C}^{n+1}$ for $n>2$ whose Levi form has a $1$-dimensional kernel at every point and that have arbitrary constant nilpotent CR symbols (defined in Remark \ref{nilpotent symbol definition}), supplying a vast set (parameterized by orbits of a Lie group acting on an infinite-dimensional vector space) of examples for each possible symbol. In particular, we obtain for every $n>2$ an {\em explicit} infinite-dimensional family of everywhere $2$-nondegenerate hypersurfaces in $\mathbb C^{n+1}$. To the best of our knowledge, this is the first such construction (where the everywhere $2$-nondegenerate hypersurfaces are described explicitely, without appealing to the space of solutions of the Levi equation). Notably, each class under discussion is a family of perturbations (by higher order terms) of the same polynomial model \eqref{model}.  Using methods that handle all cases with $n>4$ simultaneously, we solve the equivalence problem for the considered structures whose symbol is represented by a single Jordan block (Theorem \ref{Holomorphic classification theorem}), classify their algebras of infinitesimal symmetries (Theorem \ref{thealgebras}), and classify the locally homogeneous structures among them (Theorem \ref{h structures theorem}).

A surprising dearth of non-homogeneous constant symbol CR structure examples (of the type considered in  \cite{porter2021absolute,sykes2023geometry}) was this work's initial motivation. Indeed, although there are general arguments ensuring existence of a rich variety of such structures -- 
e.g., such variety becomes apparent when contrasting the main results of \cite{fels2008classification} and \cite{kolar2019complete}
-- relatively few hypersurface realizations of such structures have been described in a closed form before.


We will now outline the paper's structure while summarizing its main results. For each $n>2$, there are archetypal hypersurfaces  in $\mathbb{C}^{n+1}$ from which we will construct this paper's examples. They have the general form of CR tubes
\begin{align}\label{FGF}
x_0=\sum_{0<j<j+k\leq n}S_{j,k}x_jx_kx_n^{n-j-k},
\end{align}
where $S_{j,k}$ are scalar components of some symmetric matrix $S$, and  where here and throughout the sequel we let $z=(z_0,\ldots, z_n)$ denote standard complex coordinates of $\mathbb{C}^{n+1}$ with real and imaginary parts given by 
\[
z_j=x_j+iy_j.
\]
 We will also later use the notation $z_j = \Re(z_j)+ i \Im(z_j)$.
Occasionally we will switch between the alternative labels of
\begin{align}\label{alternate notation}
w=u+iv=z_0 
\quad\mbox{ and }
\zeta=s+it=z_n,
\end{align}
to emphasize the distinguished geometric significance of these variables $w$ and $\zeta$, namely, that the lines and planes parameterized by $v$ and $\zeta$ will be respectively transversal to the considered hypersurface and tangent to its Levi kernel at the origin.

Formula \eqref{FGF} does not, however, define a Levi degenerate hypersurface for generic $S$, so we seek values of $S$ that make \eqref{FGF} both Levi degenerate and holomorphically nondegenerate, which for the hypersurfaces of the form in \eqref{FGF} is equivalent to $2$-nondegeneracy. The general solution for such $S$ is obtained in Section \ref{Uniqueness of the defining equation formula} (Theorem \ref{solution to S theorem}), but first let us consider one strikingly simple solution, given by taking $S_{j,k}=1$ for all $(j,k)$.
\begin{theorem}\label{main theorem}
For every integer $n>2$ and twice differentiable function $f:\R\to \R$, the hypersurface $M_f$ in $\mathbb{C}^{n+1}$  defined by
\begin{align}\label{main formula}
x_0=f(x_1)+\sum_{0<j<j+k\leq n}x_jx_kx_n^{n-j-k}
\end{align}
is uniformly $2$-nondegenerate, has a $1$-dimensional Levi kernel, and at every point its Levi kernel adjoint operators (local invariants defined in \eqref{LKAO def} below) are spanned by  a nilpotent operator with a $1$-dimensional eigenspace.
\end{theorem}

This theorem is proved in Section \ref{Proof of main theorem}, following necessary preliminaries introduced in Section \ref{Preliminaries}. In Section \ref{Linking and extending defining forms}, we established a useful generalization of Theorem \ref{main theorem}, namely Theorem \ref{second main theorem}, formulated in terms of two constructions of hypersurfaces that we introduce and term \emph{linkings} and \emph{extensions}. In Section \ref{Uniqueness of the defining equation formula} we classify, up to holomorphic changes of variables, all solutions for $S$ such that \eqref{FGF} defines a uniformly $2$-nondegenerate hypersurface with a rank $1$ Levi kernel. We show, in particular, that all such solutions yield hypersurface structures equivalent to ones described by Theorem \ref{second main theorem}. The aforementioned solution with $S_{j,k}=1$ for all $(j,k)$ turns out to be the most fundamental solution among these, however, because any of the other solutions can be constructed from lower-dimensional hypersurfaces of the form in \eqref{main formula} through a combination of the linking and extension constructions introduced in Section \ref{Linking and extending defining forms}. 

Consequently, we refer to the hypersurface in \eqref{FGF} with $S_{j,k}=1$ for all $(j,k)$ as the \emph{indecomposable model} in $\mathbb{C}^{n+1}$. The indecomposable model's CR structure is equivalent to a left-invariant structure on a nilpotent Lie group. We show this in Section \ref{Left-invariant structures on Lie groups} and describe their symmetry algebras in Theorem \ref{modelalgebra} and Remark \ref{modelalgebra remark}. The lowest dimension in which these indecomposable models are defined is $\mathbb{C}^4$, and the model in $\mathbb{C}^4$ has been studied in several contexts, including \cite[Section 6]{belweighted}, \cite[Section 5, Example 1 with $n=3$]{labovskii1997dimensions}, \cite[Theorem 2]{mozey2000}, \cite[formula (1.7) with $n=3$]{porter2021absolute}, and \cite[column 4 in Table 11]{santi2020homogeneous}. Although they are in many respects natural generalizations of the well known model in $\mathbb{C}^4$, to the authors' knowledge, the higher-dimensional indecomposable models have not been studied before.

From now on we consider the $C^\infty$ smooth CR structures described by Theorem \ref{main theorem}. 
The next result obtained is the following classification of these structures.
We present its proof in Section \ref{Holomorphic classification}.
\begin{theorem}\label{Holomorphic classification theorem}For all $n>4$, the hypersurfaces $M_f$ and $M_{f^*}$ in $\mathbb{C}^{n+1}$ corresponding to two {smooth} functions $f$ and $f^*$ with $f(0)=f^*(0)=0$ are equivalent as smooth CR structures at the origin if and only if the fourth order derivatives satisfy 
\[
f^{(4)}(x)=c_1(f^*)^{(4)}(c_2x)
\]
for some constants $c_1,c_2$ and all $x$ sufficiently near the origin.
\end{theorem}

We shall note at this point that since the Levi form of each hypersurface \eqref{main formula} has two eigenvalues of different sign, any smooth CR-diffeomorphisms of such hypersurfaces are actually biholomorphisms (see, e.g., \cite{ber}) and hence all notions of CR equivalence simply  coincide in our case.

\begin{corollary}\label{Holomorphic classification corollary}For all $n>4$, the hypersurfaces $M_f$ and $M_{f^*}$ in $\mathbb{C}^{n+1}$ corresponding to two {smooth} functions $f$ and $f^*$ are equivalent near respective points $(z_0^\prime,\ldots, z_n^\prime)\in M_f$ and $(z_0^{\prime\prime},\ldots, z_n^{\prime\prime})\in M_f^*$ if and only if the fourth derivatives satisfy $$f^{(4)}(x)=c_1(f^*)^{(4)}(c_2x-c_2x_1^{\prime\prime}+x_1^{\prime})$$ for some constants $c_1,c_2$ and all $x$ sufficiently near $x_1^{\prime\prime}$.
\end{corollary}

This shows, in particular, that the moduli space of structures described in Theorem \ref{main theorem} (and its generalization Theorem \ref{second main theorem}) is {infinite dimensional}. Since \eqref{main formula} accordingly describes a rich class of structures, it is natural to ask which among them are homogeneous. As concluding results, we derive in Section \ref{Classification of the infinitesimal automorphism algebras} descriptions of the infinitesimal symmetry algebras of $M_f$ at arbitrary points for arbitrary {smooth} $f$ (Theorem \ref{thealgebras}). As an immediate corollary of Theorem \ref{thealgebras}, we obtain the following classification of the homogeneous structures described in Theorem \ref{main theorem}. 
\begin{theorem}\label{h structures theorem}
The hypersurface $M_f$ in $\mathbb{C}^{n+1}$ with $n>4$ described by \eqref{main formula} is homogeneous near the origin $z=0$ if and only if
\begin{align}\label{homogeneity criterion ODE}
f(0)=0 \quad\mbox{ and }\quad  yy^{\prime\prime}=c\left(y^{\prime}\right)^2 
\quad\mbox{ with $y=f^{(4)}(x)$ for some }c\in\mathbb{R}.
\end{align}
Up to the equivalence relation given in Theorem \ref{Holomorphic classification theorem}, $f$ satisfies the ODE \eqref{homogeneity criterion ODE} if and only if it has one of the following forms:
\begin{enumerate}[label={\bf Type \Roman*:}, leftmargin=3cm]
     \item $\quad f(x)=(x+1)^a-ax-1 \quad\quad \mbox{ for some }a\in \mathbb{R}$,
     \item $\quad f(x)=e^x-x-1$,
     \item $\quad f(x)=\ln(x+1)-x$,
     \item $\quad f(x)=(x+1)\ln(x+1)-x$,
     \item $\quad f(x)=2(x+1)^2\ln(x+1)-(x+1)^2+1$, or
     \item $\quad f(x)=6(x+1)^3\ln(x+1)-2(x+1)^3+2$.
\end{enumerate}
Different choices of $f$ in this list  correspond to different non-equivalent structures, with the one exception that $f(x)=(x+1)^a-ax-1$ corresponds to an equivalent structure for each $a$ in $\{0,1,2,3\}$.
\end{theorem}
Theorem \ref{homogeneity criterion ODE} is a generalization of \cite[Theorem 10.6]{sykes2023geometry}, which covers the special case in $\mathbb{C}^6$ (i.e., corresponding to $n=5$).

It is notable that the defining equations obtained in Theorem \ref{h structures theorem} are real-analytic, corresponding to the well known result from Lie theory that locally homogeneous CR hypersurfaces are real-analytic (as noted in \cite{zaitsev2007untitled}).
The ODE in \eqref{homogeneity criterion ODE} is particularly well studied for $c=\tfrac{3}{2}$ as this is simply the condition that the Schwarzian derivative of $f^{\prime\prime\prime}$ is zero. One solution to \eqref{homogeneity criterion ODE} with $c=\tfrac{3}{2}$ is the Type V formula of Theorem \ref{h structures theorem}.

Section \ref{Classification of the infinitesimal automorphism algebras} gives a complete classification of the infinitesimal automorphism algebras of hypersurfaces $M_f$. We give here an abreviated non-technical formulation of this result, while a detailed classification can be found in Theorem \ref{thealgebras}.

\begin{theorem}\label{algebra-description}
The dimensions of the infinitisimal automorphism algebras of hypersurfaces $M_f$ in $\mathbb{C}^{n+1}$ with $n>4$ given by \eqref{main formula} satisfy
$$2n+1\leq \mbox{dim}\,\mathfrak{hol}\,(M_f,0)\leq 2n+4,$$
and all the dimensions {between these bounds} are realized by actual hypersurfaces. For each of the dimensions a complete list of hypersurfaces $M_f$ realizing it is described in Theorem \ref{thealgebras} below.
\end{theorem}

Observe that the bounds in Theorem \ref{algebra-description} improve the general bounds known up to now (see, e.g., \cite{lm22}) for our particular class of hypersurfaces.

In this paper we do not address the local equivalence problem for the homogeneous structure classifications for the exceptional cases of $n=3$ and $n=4$, as they require a significant detour through branching analysis that distracts from this paper's main aim, the simultaneous treatment for all higher-dimensional cases, and these exceptional cases can anyway be treated via methods of \cite{porter2021absolute,sykes2023geometry} (see Remark \ref{homogeneous structures remark}).

\begin{remark}\label{homogeneous structures remark}
In \cite[Section 10]{sykes2023geometry}, a method developed therein is applied to fully solve the local equivalence problem for hypersurfaces of Theorem \ref{main theorem} in $\mathbb{C}^6$ (i.e., with $n=5$) and classify the homogeneous structures, deriving special cases of Theorems \ref{Holomorphic classification theorem} and \ref{h structures theorem}. In fact, such analysis is where the formulas of Theorem \ref{h structures theorem} originated, as it is shown in \cite{sykes2023geometry} that \eqref{homogeneity criterion ODE} is indeed a criterion for homogeneity in $\mathbb{C}^6$. The alternative analysis here, contrastingly, yields a different ODE criterion, but with exactly the same solution set. The method of \cite{sykes2023geometry} can be applied to solve the equivalence problem and classify the homogeneous structures in $\mathbb{C}^{n+1}$ of type \eqref{main formula} for any fixed dimension $n>3$, but it cannot straightforwardly address this problem for arbitrary $n$. Similarly, the bi-graded prolongation construction of \cite{porter2021absolute} can be applied to these problems for $n=3$ (and $n=4$). All cases with $n>4$ are simultaneously treated in the sequel.
\end{remark}

As a concluding note, defining equations of Theorems \ref{main theorem} and \ref{second main theorem} are but early examples of large classes of constant CR symbol $2$-nondegenerate hypersurfaces. We expect that many more such examples exist, although deriving their closed form descriptions will be a challenge as they arise from solutions to rather complicated overdetermined PDE systems. For $k>2$, there are examples of homogeneous structures (e.g., \cite{fels2006locally,fels2008classification,FK11,kruglikov20233I,labovskii1997dimensions,marini2021higher}) and the recent study \cite{kruglikov20233II} of non-homogeneous $3$-nondegenerate structures in $\mathbb{C}^4$. Given the early state of this $k>2$ frontier, the analogous open problems of finding closed form defining equations are interesting for higher orders of $k$-nondegeneracy.

\section{Preliminaries}\label{Preliminaries}

Let us recall some preliminary definitions, formulated for hypersurfaces of the form 
\[
M=\{z\in \mathbb{C}^{n+1}\,|\, z_0+\overline{z_0}=\Phi(z_1,\ldots, z_n,\overline{z_1},\ldots, \overline{z_n})\}
\]
for some smooth real-valued function $\Phi$ depending only on the real parts $x_j$ of the variables $z_j$.
Consider the complex vector fields $X_j$ taking values in the holomorphic tangent bundle $T^{1,0}\mathbb{C}^{n+1}$ given by
\begin{align}\label{TCR bundle}
X_j=\frac{\partial}{\partial z_j}+\frac{\partial \Phi}{\partial z_j}\frac{\partial}{\partial z_0}
\quad\quad \forall\, j\in\{1,\ldots, n\},
\end{align}
where 
\[
\frac{\partial}{\partial z_j}:=\frac{1}{2}\left(\frac{\partial}{\partial x_j}-i\frac{\partial}{\partial y_j}\right)
\quad\quad \forall\, j\in\{1,\ldots, n\}.
\]
The tangent bundle of $M$ has the form
\[
TM=\mathrm{span}_{\R}\left\{\frac{\partial}{\partial y_0},X_1+\overline{X_1},\ldots,X_n+\overline{X_n},iX_1-i\overline{X_1},\ldots, iX_n-i\overline{X_n}\right\},
\]
and the complex tangent bundle of the CR structure on $M$ is 
\[
H:=\mathrm{span}_{\mathbb{C}}\left\{  X_1,\ldots, X_n\right\}.
\]

Since $\Phi$ is real valued and depending only on the $x_j$ variables
\[
\left[X_j,\overline{X_k}\right]=\frac{\partial^2 \Phi}{\partial z_j\partial \overline{z_k}} \left( \frac{\partial}{\partial \overline{z_0}}-\frac{\partial}{\partial z_0}\right)=
\frac{i}{4}\frac{\partial^2 \Phi}{\partial x_j\partial x_k}  \frac{\partial}{\partial y_0}
\quad\quad \forall\, j,k\in\{1,\ldots, n\},
\]
and hence the Levi form $\mathcal{L}$ of  $M$ is represented with respect to the respective bases $(X_1,\ldots, X_n)$  and $ \frac{i}{4}\tfrac{\partial}{\partial y_0}$ by the Hermitian (and even real symmetric, due to the hypersurface's tubularity) matrix
\begin{align}\label{levi form}
H_{\mathcal{L}}:=
\left(
\begin{array}{ccc}
\frac{\partial^2 \Phi}{\partial x_1^2} & \cdots & \frac{\partial^2 \Phi}{\partial x_1\partial x_n}  \\
\vdots & \ddots & \vdots \\
\frac{\partial^2 \Phi}{\partial x_n\partial x_1} & \cdots & \frac{\partial^2 \Phi}{\partial x_n^2} 
\end{array}
\right).
\end{align}

Suppose furthermore that $H_{\mathcal{L}}$ has rank $n-1$ at every point, and let $K$ denote its kernel, that is, $K$ is the complex line bundle on $M$ whose fiber $K_q$ at a point $q\in M$ is defined by 
\[
K_q:=\{v\in H_q\,|\, \mathcal{L}(v,w)=0\quad\forall\, w\in H_q\}.
\]
Following its description in \cite{porter2021absolute}, for any $q\in M$ and vector $v\in K_q$, we define the associated \emph{Levi-kernel adjoint operator} $\mathrm{ad}_{v}:H_q/K_q\to H_q/K_q$ by taking any vector field $V\in \Gamma(K)$ with $V_q=v$ and define
\begin{align}\label{LKAO def}
\mathrm{ad}_v(X_q+K_q):=\left[V,\overline{X}\right]_q\pmod{\overline{H}\oplus K}\quad\quad\forall X\in \Gamma(H),
\end{align}
where $\Gamma(\cdot)$ denotes sections of a fiber bundle. It is an exercise (analogous to deriving the standard Lie bracket definition of the Levi form) to show that \eqref{LKAO def} depends only on the value of $X$ at $q$. The set $\{\mathrm{ad}_v\,|\, v\in K_q\}$ is a local invariant of the CR structure at $q\in M$.
\begin{remark}[Terminology related to the $\mathrm{ad}_v$ operators]\label{nilpotent symbol definition}
    The CR symbol of $M$ at $q$ is a local invariant defined in \cite{porter2021absolute} in terms of the Levi form at $q$ and the set $\{\mathrm{ad}_v\,|\, v\in K_q\}$. Matching the terminology of \cite{porter2021absolute}, we say that the CR symbol is nilpotent if each $\mathrm{ad}_v$ is nilpotent.

    Structures whose Levi form's kernel is a regular distribution are $2$-nondegenerate if and only if the mapping $v\mapsto \mathrm{ad}_v$ is injective. While this suffices as a definition in our present specialized setting, we refer to \cite{ber} for the broader definitions of $2$-nondegeneracy and more generally $k$-nondegeneracy.
\end{remark}

We will describe now the Levi kernel adjoint operators in terms of $\Phi$. For this suppose that $X_1,\ldots, X_{n-1}$ are all not in the Levi kernel at every point, which can be achieved at least locally without loss of generality by permuting the coordinates because we are also assuming that $\mathcal{L}$ has rank $n-1$. Let $v$ be a nonzero $M$-dependent column vector whose scalar components are functions of $(x_1,\ldots, x_n)$ such that $H_{\mathcal{L}}v=v^TH_{\mathcal{L}}=0$ for every point in $M$. Accordingly,
\begin{align}\label{lk formula}
V:=\sum_{j=1}^n v_j X_j
\end{align}
is a vector field spanning the Levi kernel at every point in $M$. For each $q\in M$, we have
\begin{align}
\mathrm{ad}_{V_q}(X_k)=\sum_{j=1}^nv_j\left[X_j,\overline{X_k}\right]-\overline{X_k}(v_j)X_j 
 = \sum_{j=1}^nv_j\left(H_{\mathcal{L}}\right)_{j,k}-\overline{X_k}(v_j)X_j 
& =- \sum_{j=1}^n\overline{X_k}(v_j)X_j \\
& =- \frac12 \sum_{j=1}^n \frac{\partial v_j}{\partial x_k}X_j,
\end{align}
where $\pmod{\overline H\oplus K}$ has been omitted. Therefore, the antilinear operator $\mathrm{ad}_{V_q}$ is represented by the matrix
\begin{align}\label{LKAO WRT Phi}
A:=\left(
\begin{array}{ccc}
\frac{\partial v_1}{\partial x_1} & \cdots & \frac{\partial v_1}{\partial x_{n-1}}  \\
\vdots & \ddots & \vdots \\
\frac{\partial v_{n-1}}{\partial x_1}& \cdots & \frac{\partial v_{n-1}}{\partial x_{n-1}} 
\end{array}
\right)
\end{align}
with respect to $(X_1,\ldots, X_{n-1})$, where $v_1,\ldots, v_n$ are functions of $(x_1,...,x_n)$ satisfying
\[
0=\sum_{j=1}^nv_j\left(H_{\mathcal{L}}\right)_{j,k}=\sum_{j=1}^nv_j \frac{\partial^2 \Phi}{\partial x_j\partial x_k}
\quad\quad\forall k\in \{1,\ldots, n\},
\]
which are uniquely determined up to a (point-dependent) rescaling of $v$.

\section{Proof of Theorem \ref{main theorem}} \label{Proof of main theorem}
Define the real valued function $\Phi$ in a neighborhood of $0\in \mathbb{C}^n$ by 
\begin{align}\label{main formula alt}
\Phi\left(z_1,\ldots, z_n,\overline{z_1},\ldots, \overline{z_n}\right):=f(x_1)+\sum_{0<j<j+k\leq n}x_jx_kx_n^{n-j-k},
\end{align}
and
\[
M:=\left\{  z\in\mathbb{C}^{n+1}\,\left|\,\Phi\left(z_1,\ldots, z_n,\overline{z_1},\ldots, \overline{z_n}\right)=z_0+\overline{z_0}\right.\right\},
\]
so that $M\subset \mathbb{C}^{n+1}$ is equivalent to the hypersurface defined by \eqref{main formula} via the biholomorphism $z\mapsto \left(\tfrac{z_0}{2}, z_1,\ldots, z_n\right)$.

Noting \eqref{levi form} and \eqref{main formula alt}, we get 
\begin{align}\label{levi form expl}
\bgroup
\arraycolsep=2pt
H_{\mathcal{L}}=
2\left(
\begin{array}{cccccc;{2pt/2pt}c}
x_n^{n-2}+f^{\prime\prime}(x_1) & x_n^{n-3} & \cdots& x_n^2 & x_n& 1 & \sum_{j=1}^{n-2}jx_{n-1-j}x_n^{j-1}\\
x_n^{n-3} & & \iddots & x_n & 1 &0& \vdots \\
\vdots&\iddots & \iddots & \iddots & 0 & 0  &\vdots\\
x_n^{2} & \iddots& \iddots & \iddots &  & \vdots & \sum_{j=1}^{2}jx_{3-j}x_n^{j-1}\\
x_n & \iddots& 0 &  &  & 0 & \sum_{j=1}^{1}jx_{2-j}x_n^{j-1}\\
1 & 0& 0 &\cdots  & 0 & 0 & 0 \\\hdashline[2pt/2pt]
\left( H_{\mathcal{L}}\right)_{n,1} &\cdots & \cdots & \cdots & \cdots & & \sum\limits_{j=1}^{n-3}\binom{j+1}{2}\left(\sum\limits_{r+s=n-1-j}x_rx_sx_n^{j-1}\right)
\end{array}
\right),
\egroup
\end{align}
and using this explicit formula for $H_{\mathcal{L}}$ one can easily calculate that $H_{\mathcal{L}}v=v^TH_{\mathcal{L}}=0$ for the column vector
\begin{align}\label{kernel of levi form expl}
v:=
\left(
\begin{array}{c}
0 \\\hdashline[2pt/2pt]
\sum_{j=1}^1 x_jx_n^{1-j} \\
\sum_{j=1}^2 x_jx_n^{2-j} \\
\vdots \\
\sum_{j=1}^{n-2} x_jx_n^{n-2-j} \\\hdashline[2pt/2pt]
-1
\end{array}
\right).
\end{align}
So the vector field
\[
V:=\sum_{j=1}^n v_j X_j
\]
is in the Levi kernel of $M$. Furthermore, $V$ spans the Levi kernel of $M$ at every point because $H_{\mathcal{L}}$ has a $1$-dimensional kernel everywhere, which can be seen by noting that the upper left $(n-1)\times (n-1)$ block in \eqref{levi form expl} is nondegenerate.

By \eqref{LKAO WRT Phi}, the Levi kernel adjoint operator $\mathrm{ad}_V$ is represented by the matrix
\begin{align}\label{LKAO calculation}
A=\sum_{j=1}^{n-2}x_n^{j-1}T^j_{n-1}
\end{align}
with respect to $(X_1, \ldots, X_{n-1})$, where here $T_k$ denotes the $k\times k$ Toeplitz matrix whose $(r,s)$ entry is given by
\begin{align}\label{fundamental toeplitz}
\left(T_k\right)_{r,s}:= \delta_{r-s,1}
=
\begin{cases}
1 &\mbox{ if } r-s=1\\
0 &\mbox{ if }r-s\neq1.
\end{cases}
\end{align}

Noting Remark \ref{nilpotent symbol definition}, this completes the proof of Theorem \ref{main theorem} because we have shown  that \eqref{main formula} has a rank $1$ Levi kernel and that its Levi kernel adjoint operators are, up to rescaling, represented by the matrix \eqref{LKAO calculation} with respect to $(X_1, \ldots X_{n-1})$, which is indeed a nilpotent matrix having a $1$-dimensional eigenspace.

\section{Linking and extending defining equations}\label{Linking and extending defining forms}

We have seen that \eqref{main formula} defines a uniformly $2$-nondegenerate hypersurface with Levi kernel adjoint operators that admit matrix representations by single nilpotent Jordan blocks. To obtain structures with arbitrary nilpotent Levi kernel adjoint operators, we introduce now processes of \emph{linking} and \emph{extending} defining equations. These processes will be applied to the above tubular hypersurfaces, but they can just as easily be defined in greater generality, which is how we present their definition. 

\begin{definition}
A \emph{$2$-dimensional extension}  $M^\prime$ of a hypersurface $M$ defined by 
\[
M:=\left\{z\in\mathbb{C}^{n+1}\,\left|\,\Re(z_0)=f\left(z_1,\ldots, z_n,\overline{z_1},\ldots, \overline{z_n}\right)\right.\right\}
\]
for some real-valued function $f$, is a hypersurface in $\mathbb{C}^{n+2}$ defined by 
\[
M^\prime=\left\{z\in\mathbb{C}^{n+2}\,\left|\,\Re(z_0)=f\left(z_1,\ldots, z_n,\overline{z_1},\ldots, \overline{z_n}\right)+\epsilon \left(\Re z_{n+1}\right)^2\right.\right\}
\]
for some $\epsilon\in \{-1,1\}$. Similarly $f+\epsilon \left(\Re z_{n+1}\right)^2$ will be called an \emph{extension of the defining form $f$}.
\end{definition}

\begin{definition}\label{link definition}
Let $M_1$ and  $M_2$ be two $2$-nondegenerate hypersurfaces $M_1\subset\mathbb{C}^{n_1+1}$ and $ M_2\subset\mathbb{C}^{n_2+1}$, $n_1,n_2\geq 1$, defined by 
\[
M_1:=\left\{z\in\mathbb{C}^{n_1+1}\,\left|\,\Re(z_0)=\Phi_1\left(z_1,\ldots, z_{n_1},\overline{z_1},\ldots, \overline{z_{n_1}}\right)\right.\right\}
\]
and
\[
M_2:=\left\{z\in\mathbb{C}^{n_2+1}\,\left|\,\Re(z_0)=\Phi_2\left(z_1,\ldots, z_{{n_2}},\overline{z_1},\ldots, \overline{z_{{n_2}}}\right)\right.\right\}
\]
for some defining functions $\Phi_1$ and $\Phi_2$ such that the vector fields $(X_1)_1,\ldots, (X_1)_{n_1-1}$ and $(X_2)_1,\ldots,(X_2)_{n_2-1}$ are not in their respective hypersurface's Levi kernel, where these vector fields are defined by \eqref{TCR bundle} while setting $\Phi=\Phi_j$ and $X_k=(X_j)_k$  for each $j=1,2$.
\emph{The link $\Phi$ of the defining functions $\Phi_1$ and $\Phi_2$} is the function $\Phi:\mathbb{C}\oplus \mathbb{C}^{n_1-1}\oplus \mathbb{C}^{n_2-1}\to \R$ given by
\begin{align}
\Phi( z_1,\ldots, z_{n_1+n_2-2},\xi,\overline{z_1},\ldots, \overline{z_{n_1+n_2-2}},\overline{\xi})&:=\Phi_1\left(z_1,\ldots, z_{n_1-1},\xi,\overline{z_1},\ldots, \overline{z_{n_1-1}},\overline{\xi}\right) \\
& \hspace{.6cm}+\Phi_2\left(z_{n_1},\ldots, z_{n_1+n_2-2},\xi,\overline{z_{n_1}},\ldots, \overline{z_{n_1+n_2-2}},\overline{\xi}\right)
\end{align}
The hypersurface $\{z\in\mathbb{C}^{n_1+n_2}\,|\, \Re(z_0)=\Phi(z_1,\ldots, z_{n_1+n_2-1},\overline{z_1},\ldots, \overline{z_{n_1+n_2-1}})\}$ will be called a \emph{link} of $M_1$ and $M_2$.
\end{definition}
\begin{lemma}\label{link of 2-nondegenerates lemma}
For $\Phi_1$, $\Phi_2$ and $\Phi$ as in Definition \ref{link definition}, the link $\Phi$ of $\Phi_1$ and $\Phi_2$ defines a uniformly $2$-nondegenerate structure.
\end{lemma}
\begin{proof}
For $j\in\{1,2\}$, let $H_j$ (respectively $H$) denote the matrix representing the Levi form of $M_j:=\{\Re(z_0)=\Phi_j\}$ (respectively $M:=\{\Re(z_0)=\Phi\}$) given by \eqref{levi form}. Let us partition $H_1$ and $H_2$ into a block decompositions with components labeled $A_j$, $B_j$, and $C_j$ such that 
\[
H_j=
\left(
\begin{array}{cc}
    A_j & B_j \\
    B_j^* & C_j
\end{array}
\right)
\]
and each $C_j$ is a $1\times 1$ matrix. Accordingly,
\begin{align}\label{linked Levi form}
H=
\left(
\begin{array}{ccc}
    A_1 & 0 & B_1 \\
    0 & A_2 & B_2 \\
    B_1^* & B_2^* & C_1+C_2
\end{array}
\right).
\end{align}
Let $v_j$ be the nonzero vector of the form 
\begin{align}\label{linked component kernel formula}
v_j=
\left(
\begin{array}{c}
    v^\prime_j \\
    1
\end{array}
\right)
\end{align}
for some length $n_j-1$ vector $v^\prime_j$, and $H_jv_j=0$, that is,
\begin{align}\label{kernel condition}
A_jv^\prime_j+B_j=0
\quad\mbox{ and }\quad
B_j^*v^\prime_j+C_j=0
\end{align}
Indeed such $v_j$ exist because the properties of $\Phi_j$ given in Definition \ref{link definition} imply that the kernel of each $H_j$ is nonempty and transversal to $\{(z_1,\ldots, z_{n_j})\,|\, z_{n_j}=0\}$. In particular, each $A_j$ is nonsingular, so the matrix in \eqref{linked Levi form} has at most a $1$-dimensional kernel. By \eqref{linked component kernel formula} and \eqref{kernel condition} the vector 
\begin{align}\label{linked kernel formula}
    v=
    \left(
\begin{array}{c}
    v^\prime_1 \\
    v^\prime_2 \\
    1
\end{array}
\right)
\end{align}
is in the right kernel of $H$ (i.e., $Hv=0$). A matrix representing one of the Levi kernel adjoint operators of $M$ (respectively $M_j$) is given by \eqref{LKAO WRT Phi} with $v$ as in \eqref{linked kernel formula} (respectively \eqref{linked component kernel formula}), from which it is clearly seen that the matrix is a direct sum of matrices representing nonzero Levi kernel adjoint operators of $M_1$ and $M_2$. Hence $M=\{\Re(z_0)=\Phi\}$ is uniformly $2$-nondegenerate with a rank 1 Levi kernel.
\end{proof}

We will describe now in more detail structures arising from linking and extending the defining equations in \eqref{main formula}.

To state the following theorem, we introduce notation, for a  function $f:\R\to\R$ and integer $n>0$, labeling as $P_{f, n}(x_1,\ldots, x_n)$ the function of $n+1$ real variables $\left(x_1,\ldots, x_n,\Re(\xi)\right)$ given by
\begin{align}\label{main defining function}
P_{f, n}(x_1,\ldots, x_n):=f(x_1)+\sum_{j+k\leq n+1}x_jx_k\left(\Re \xi\right)^{n+1-j-k},
\end{align}
where $0<j, 0<k$.
We also index, for a sequence of positive integers $n_1,\ldots, n_\mu$, the complex coordinates of $\mathbb{C}^2\oplus\mathbb{C}^{n_1}\oplus\cdots \oplus\mathbb{C}^{n_\mu}$ by
\[
(z_0,\xi,z_{1,1},\ldots, z_{1,n_1},\ldots, ,z_{\mu,1},\ldots, z_{\mu,n_\mu}).
\]

\begin{theorem}\label{second main theorem}
For any nilpotent matrix $A$, by linking and extending defining equations of the type in \eqref{main formula}, one can obtain defining equations of uniformly $2$-nondegenerate hypersurfaces with rank $1$ Levi kernel having Levi kernel adjoint operators represented by $A$ at every point. In particular for any twice differentiable functions $f_1,\ldots, f_\mu$ and positive integers $n_1,\ldots, n_\mu$, the hypersurface
\[
M:=\{z\in \mathbb{C}^2\oplus\mathbb{C}^{n_1}\oplus\cdots \oplus\mathbb{C}^{n_\mu}\,|\, \Re(z_0)=\Phi \}
\]
defined with
\[
\Phi\left(z_{1,1},\ldots, z_{\mu,n_\mu},\xi,\overline{z_{1,1}},\ldots, \overline{z_{\mu,n_\mu}},\overline{\xi}\right):=\sum_{j=1}^\mu P_{f_j,n_j}(\Re z_{j,1},\ldots, \Re z_{j,n_j})
\]
has Levi kernel adjoint operators at every point represented by the nilpotent matrix
\begin{align}\label{nilpotent sum}
A=T_{n_1}\oplus\cdots \oplus T_{n_\mu}
\end{align}
where $P_{f_j,n_j}$ are defined in \eqref{main defining function} and $T_k$ are the nilpotent matrices defined in \eqref{fundamental toeplitz}.
\end{theorem}

We omit a detailed proof of Theorem \ref{second main theorem} because it is essentially outlined in the proof of Lemma \ref{link of 2-nondegenerates lemma} and Theorem \ref{main theorem} detailed already above.  The only meaningful difference between this and Theorem \ref{main theorem} is that instead of having \eqref{levi form expl}, the upper left $(n-1)\times (n-1)$ block of $H_{\mathcal{L}}$ in \eqref{levi form expl} should be replaced with a block diagonal matrix having $\mu$ blocks on the diagonal of sizes $n_1\times n_1,\ldots, n_\mu\times n_\mu$ (e.g., as in \eqref{linked Levi form} for the case $\mu=2$), of which each has the same general form as the upper left $(n-1)\times (n-1)$ in \eqref{levi form expl}. And, similarly, the last column and last row of $H_{\mathcal{L}}$ change from the formula in \eqref{levi form expl} as follows: one replaces the upper right $(n-1)\times 1$ block (respectively, lower left $1\times (n-1)$ block) with $\mu$ blocks of sizes $n_1\times 1,\ldots, n_\mu\times 1$ (respectively, $1\times n_1,\ldots , 1\times n_\mu$), whose exact formulas are straightforward to calculate.  Noting that \eqref{kernel of levi form expl} defines a vector in the kernel of the matrix in \eqref{levi form expl}, and noting the similarities between this new $H_{\mathcal{L}}$ corresponding to Theorem \ref{second main theorem} and the matrix in \eqref{levi form expl}, one can easily deduce the general form of a field of vectors $v$ in the kernel of this new $H_{\mathcal{L}}$. Lastly, applying \eqref{LKAO WRT Phi} to any such nonzero $v$, one finds directly that $A$ is a direct sum of nilpotent  Jordan blocks with the appropriate sizes, implying that the Levi kernel adjoint operators can each be represented by \eqref{nilpotent sum} with respect to some basis.

\section{2-nondegeneracy and formula (\ref{FGF})}\label{Uniqueness of the defining equation formula}

In this section we derive Theorem \eqref{solution to S theorem}
classifying, up to local equivalence, all uniformly $2$-nondegenerate hypersurfaces with rank $1$ Levi kernels defined by \eqref{FGF}.

\begin{lemma}\label{normalizing quartic terms}
Every {smooth} hypersurface $M\subset \mathbb{C}^{n+1}$ of the type described in Theorem \ref{main theorem} given by \eqref{main formula} is equivalent via a holomorphic change of coordinates to a hypersurface of the form given by \eqref{main formula} with $f(0)=f^\prime(0)=f^{\prime\prime}(0)=f^{(3)}(0)=0$ and $f^{(4)}(0)\in\{0,1\}$.
\end{lemma} 

\begin{proof}
Let
\begin{align}\label{homogeneous defining function}
F(z)=f(x_1)+\sum_{0<j<j+k\leq n}x_jx_kx_n^{n-j-k}.
\end{align}
Firstly, 
\begin{equation}\label{remove}
z\mapsto \left(z_0+f(0)-\frac{f^{(3)}(0)}{12}z_1^3,z_1,\ldots, z_{n-1}-\frac{4f^\prime(0)+2f^{\prime\prime}(0)z_1+f^{(3)}(0)z_1^2}{8},z_n\right)
\end{equation} 
transforms $x_0=F$ into $x_0=F-f(0)-f^\prime(0)x_1-\tfrac{f^{\prime\prime}(0)}{2}x_1^2-\tfrac{f^{(3)}(0)}{6}x_1^3$, so indeed we can achieve $f(0)=f^\prime(0)=f^{\prime\prime}(0)=f^{(3)}(0)=0$.

Now to normalize $f^{(4)}(0)$, consider the weight system
\[
w(x_0)=5,\quad
w(x_n)=\frac{3}{n-2},
\quad\mbox{ and }\quad
w(x_j):=\frac{n-5+3j }{n-2}
\quad\forall\, j\in\{1,\ldots, n-1\},
\]
assigning weights to the variables $x_0,\ldots, x_n$.
If $f^{(4)}(0)\neq 0$, then the corresponding weighted dilation
\[
x\mapsto \left(\left(f^{(4)}(0)\right)^{w(x_0)}x_0,\left(f^{(4)}(0)\right)^{w(x_1)}x_1,\ldots, \left(f^{(4)}(0)\right)^{w(x_n)} x_n\right),
\]
transforms $x_0=F$  into
\[
x_0=f\left(f^{(4)}(0)x_1\right)\left(f^{(4)}(0)\right)^{-5}+\sum_{0<j<j+k\leq n}x_jx_kx_n^{n-j-k},
\] 
normalizing $f^{(4)}(0)$ because
\[
\left.\frac{d^4}{dx_1^4}\left( f\left(f^{(4)}(0)x_1\right)\right)\left(f^{(4)}(0)\right)^{-5}\right|_{x_1=0}=1.
\]
\end{proof}

\begin{lemma}\label{solutions to S lemma}
If $M\subset \mathbb{C}^{n+1}$ is the hypersurface defined by \eqref{FGF} with
\begin{align}\label{general coeficient matrix}
S=
\left(
\begin{array}{ccccc}
    0 & \ldots & \ldots & 0 & S_{n_\mu}\\
    \vdots &  & \iddots & S_{n_{\mu-1}} & 0\\
    \vdots & \iddots & \iddots & \iddots & \vdots\\
    0 & S_{n_2} & \iddots &  & \vdots\\
    S_{n_1} & 0 & \ldots & \ldots & 0
\end{array}
\right)
\end{align}
for some positive integers $n_1,\ldots, n_\mu$ satisfying $n_1+\cdots+n_\mu=n-1$ and $n_j=n_{\mu+1-j}$ for all $j$, where $S_{l}$ denotes the $l\times l$ matrix whose $(j,k)$ entry is
\[
(S_l)_{j,k}:=
\begin{cases}
1 & \mbox{ if }j+k\leq l+1\\
0 & \mbox{ otherwise,}
\end{cases}
\]
then either 
\begin{itemize}
\item $\mu=1$ and $M$ is the $2$-nondegenerate hypersurface defined by a special case of the general form \eqref{main formula}, or
\item $\mu>1$ and $M$ is the $2$-nondegenerate hypersurface defined by a combination of links and extensions of lower-dimensional defining functions also of the form in \eqref{main formula}, as defined in Section \ref{Linking and extending defining forms}.
\end{itemize}
\end{lemma}
\begin{proof}
Clearly if $\mu=1$ then $M$ is of the form in \eqref{main formula}, so we will suppose that $\mu>1$.

We are going to define now a second hypersurface defined by a combination of linking and extending defining equations and prove that it is equivalent to the hypersurface $M$ defined by \eqref{FGF} with \eqref{general coeficient matrix}. For each $j,k\in\mathbb{N}$, let $Q_{j;k}$ denote the $j\times j$ matrix whose $(r,s)$ entry is
\[
\left(Q_{j;k}\right)_{r,s}=\delta_{r+s,j+2-k}.
\]
Define the $(n-1)\times (n-1)$ matrices ${Q}_1,\ldots,{Q}_{n-2}$ by
\[
{Q}_j:=
\left(
\begin{array}{ccccc}
    0 & \ldots & \ldots & 0 & Q_{n_\mu;j}\\
    \vdots &  & \iddots & Q_{n_{\mu-1};j} & 0\\
    \vdots & \iddots & \iddots & \iddots & \vdots\\
    0 & Q_{n_2;j} & \iddots &  & \vdots\\
    Q_{n_1;j} & 0 & \ldots & \ldots & 0
\end{array}
\right)
\]
Notice that 
\begin{align}\label{quadric representation of F}
F(z):=\sum_{j=1}^{n-1}
\left[(x_1,\ldots,x_{n-1}) Q_{j}
\left(\begin{array}{c}
    x_1  \\
    \vdots \\
    x_{n-1}
\end{array}
\right)
\right]x_n^{j-1}
\end{align}
is exactly the right side of \eqref{FGF} with $S$ as in \eqref{general coeficient matrix}, so $M$ is defined by $x_0=F(z)$.

For $j\in\{1,\ldots,\mu\}$ define $F_j:\mathbb{C}^{n_j+1}\to \mathbb{R}$ by
\[
F_j(z_1,\ldots, z_{n_j},\xi):=\sum_{j=1}^{n_j}
\left[(x_1,\ldots,x_{n_j}) Q_{n_j;k}
\left(\begin{array}{c}
    x_1  \\
    \vdots \\
    x_{n_j}
\end{array}
\right)
\right](\Re\xi)^{j-1}.
\]
Relabeling the coordinates $(z_0,\ldots, z_n)$ as
\[
(z_0,\ldots, z_n)=(z_0,z_{1,1},\ldots,z_{1,n_1},\ldots,z_{\mu,1},\ldots,z_{\mu,n_{\mu}}, z_n),
\] define $\widehat{F}:\mathbb{C}^{n+1}\to \mathbb{R}$ by
\begin{align}\label{F hat definition}
\widehat{F}(z):=\left(\sum_{j=1}^{\lceil \mu/2\rceil}F_j(x_{j,1},\ldots, x_{j,n_j},x_n)\right)-\sum_{j=\lceil \mu/2\rceil+1}^{\mu}F_j(x_{j,1},\ldots, x_{j,n_j},x_n).
\end{align}
The function $\widehat{F}$ is a combination of links and extensions of defining functions of the form described in Theorem \ref{second main theorem}, and hence $x_0=\widehat{F}$ defines a uniformly $2$-nondegenerate hypersurface in $\mathbb{C}^{n+1}$. Notice that coeffiecients of powers of $x_n$ in $\widehat{F}$ are all quadratic functions of $(x_1,\ldots, x_{n-1})$, and hence there exist symmetric matrices $\widehat{Q}_{1},\ldots \widehat{Q}_{n-1}$ such that 
\begin{align}\label{quadric representation of F hat}
\widehat{F}(z):=\sum_{j=1}^{n-1}
\left[(x_1,\ldots,x_{n-1}) \widehat{Q}_{j}
\left(\begin{array}{c}
    x_1  \\
    \vdots \\
    x_{n-1}
\end{array}
\right)
\right]x_n^{j-1},
\end{align}
where indeed we only need to consider $n-1$ such $\widehat{Q}_j$ matrices because the degree of $\widehat{F}$ with respect to $x_n$ is at most $n-2$.
Comparing \eqref{F hat definition} and \eqref{quadric representation of F hat} yields
\[
\widehat{Q}_j:=
\left(
\begin{array}{cccccc}
    Q_{n_1;j} & 0 & \cdots & \cdots & \cdots & 0\\
    0 & \ddots & \ddots &  &  & \vdots\\
    \vdots & \ddots & Q_{n_{\lceil\mu/2\rceil};j} & \ddots &  & \vdots\\
    \vdots &  & \ddots & -Q_{n_{\lceil\mu/2\rceil}+1;j} & \ddots & \vdots\\
    \vdots &  &   & \ddots & \ddots & 0 \\
    0 & \cdots & \cdots & \cdots & 0 & -Q_{n_{\mu};j}
\end{array}
\right).
\]
The matrix $X$ given by
\[
X:=
\frac{1}{\sqrt{2}}
\left(
\begin{array}{cccccc}
    \mathrm{Id}(n_1) & 0 & \cdots & \cdots & 0 &  \mathrm{Id(n_1)}\\
    0 & \ddots &  &  & \iddots & 0\\
    \vdots &  & \mathrm{Id}(n_{\mu/2}) & \mathrm{Id}(n_{\mu/2}) & & \vdots\\
    \vdots &  & \mathrm{Id}(n_{\mu/2}) & -\mathrm{Id}(n_{\mu/2}) & & \vdots\\
    0 & \iddots &   &   & \ddots& 0\\
    \mathrm{Id(n_1)} & 0 & \cdots & \cdots & 0 & -\mathrm{Id(n_1)}
\end{array}
\right)
\]
or, depending on the parity of $\mu$,
\[
X:=
\frac{1}{\sqrt{2}}
\left(
\begin{array}{ccccccc}
    \mathrm{Id}(n_1) & 0 & \cdots & \cdots & \cdots & 0 &  \mathrm{Id(n_1)}\\
    0 & \ddots &  & &  & \iddots & 0\\
    \vdots &  & \mathrm{Id}(n_{(\mu-1)/2}) & 0 & \mathrm{Id}(n_{(\mu-1)/2}) & & \vdots\\
    \vdots &  & 0  & \sqrt{2}\mathrm{Id}(n_{(\mu+1)/2}) & 0  & &  \vdots\\
    \vdots &  & \mathrm{Id}(n_{(\mu-1)/2}) & 0 & -\mathrm{Id}(n_{(\mu-1)/2}) & & \vdots\\
    0 & \iddots &   & &   & \ddots& 0\\
    \mathrm{Id(n_1)} & 0 & \cdots & \cdots & \cdots & 0 & -\mathrm{Id(n_1)}
\end{array}
\right),
\]
where $\mathrm{Id}(k)$ denotes the $k\times k$ identity matrix, satisfies
\[
Q_j=X^T\widehat{Q}_jX
\quad\quad\forall\, j.
\]
Hence, by \eqref{quadric representation of F} and \eqref{quadric representation of F hat},
if 
\[
\Psi(z_0,\ldots, z_n):=
\left(
\begin{array}{ccc}
    1 & 0 & 0 \\
    0 & X & 0 \\
    0 & 0 & 1
\end{array}
\right)
\left(
\begin{array}{c}
    z_0 \\
    \vdots \\
    z_n
\end{array}
\right)
\]
then
\[
\widehat{F}\big(\Psi(z)\big)= \sum_{j=1}^{n-1}
\left[(x_1,\ldots,x_{n-1}) X^{T}\widehat{Q}_{j}X
\left(\begin{array}{c}
    x_1  \\
    \vdots \\
    x_{n-1}
\end{array}
\right)
\right]x_n^{j-1}
=F(z).
\]
Therefore $F$ and $\widehat{F}$ define equivalent hypersurfaces, having a $2$-nondegenerate structure by Theorem \ref{second main theorem}.
\end{proof}

Lemma \ref{solutions to S lemma} actually describes all uniformly $2$-nondegenerate structures with a rank $1$ Levi kernel given by \eqref{FGF}, which is the content of the next Theorem.
\begin{theorem}\label{solution to S theorem}
A hypersruface defined by \eqref{FGF} is uniformly $2$-nondegenerate with a rank $1$ Levi kernel if and only if it
is equivalent to a hypersurface defined by \eqref{FGF} with $S$ having the form \eqref{general coeficient matrix}. In particular, \eqref{FGF} describes, up to local equivalence, at most $2^{p}-1$ uniformly $2$-nondegenerate hypersurfaces in $\mathbb{C}^{n+1}$ having rank $1$ Levi kernels, where $p=\left\lceil \tfrac{n-2}{2}\right\rceil$.
\end{theorem}
\begin{proof}

Let  $S$ be fixed such that \eqref{FGF} defines a uniformly $2$-nondegenerate hypersurface with a rank 1 Levi kernel. The terms in the set
\begin{align}\label{first set to normalize}
\left\{S_{j,k}x_jx_k\,|\, n-1\leq j+k\leq n,\, j\leq k\right\}
\end{align}
are normalized such that the coefficients $S_{j,k}$ belong to $\{0,1\}$ under the holomorphic change of variables $z\mapsto (c_0z_0,\ldots, c_{n-1}z_{n-1}, z_n)$ where the constant coefficients $\{c_j\}$ are defined recursively as follows. Define $\sigma:\{1,\ldots, n-1\}\to\{1,\ldots, n-1\}$ to be the permutation given by 
\begin{align}
\sigma\left(j\right):=
\left\lfloor\tfrac{n}{2}\right\rfloor-(-1)^{n+j}\left\lfloor \tfrac{j}{2} \right\rfloor\quad\quad\forall\, j\in\{1,\ldots, n-1\}.
\end{align}
Notice that the set in \eqref{first set to normalize} has the convenient ordering
\[
\left(S_{\sigma(1),\sigma(1)}x_{\sigma(1)}x_{\sigma(1)},S_{\sigma(1),\sigma(2)}x_{\sigma(1)}x_{\sigma(2)},\ldots, S_{\sigma(n-2),\sigma(n-1)}x_{\sigma(n-2)}x_{\sigma(n-1)}\right).
\]
Set
\[
c_0:=
\begin{cases}
1 & \mbox{ if }S_{\left\lfloor\tfrac{n}{2}\right\rfloor,\left\lfloor\tfrac{n}{2}\right\rfloor}\geq 0 \\
-1 & \mbox{ if }S_{\left\lfloor\tfrac{n}{2}\right\rfloor,\left\lfloor\tfrac{n}{2}\right\rfloor}< 0
\end{cases}
\quad\mbox{ and }\quad
c_{\sigma(1)}:=
\begin{cases}
1 & \mbox{ if }S_{\left\lfloor\tfrac{n}{2}\right\rfloor,\left\lfloor\tfrac{n}{2}\right\rfloor}= 0 \\
\left|S_{\left\lfloor\tfrac{n}{2}\right\rfloor,\left\lfloor\tfrac{n}{2}\right\rfloor}\right|^{\tfrac{-1}{2}} & \mbox{ if }S_{\left\lfloor\tfrac{n}{2}\right\rfloor,\left\lfloor\tfrac{n}{2}\right\rfloor}\neq 0
\end{cases}
\]
and, after defining $c_{\sigma(1)},\ldots, c_{\sigma(j-1)}$ for some $j\in\{2,\ldots, n-1\}$, define
\[
c_{\sigma(j)}:=
\begin{cases}
1 & \mbox{ if }S_{\sigma(j-1),\sigma(j)}= 0 \\
\frac{1}{c_{\sigma(j-1)}S_{\sigma(j-1),\sigma(j)}} & \mbox{ if }S_{\sigma(j-1),\sigma(j)}\neq 0.
\end{cases}
\]
Notice that this choice of $c_0$ normalizes the $S_{\sigma(1),\sigma(1)}$ coefficient making it non-negative, and indeed the subsequent choices of $c_1,\ldots,c_{n-1}$ ensure 
\begin{align}\label{bd normalization a}
S_{j,k}\in\{0,1\}\quad\quad\forall\,n-1\leq j+k\leq n.
\end{align}
Furthermore,
\begin{align}\label{bd normalization b}
S_{j,k}=1\quad\quad\forall\  j+k= n
\end{align}
because we are assuming that the Levi kernel is $1$-dimensional everywhere, and having $S_{j,k}=0$ for some $j+k=n$ would result in a higher-dimensional kernel at the point $z=0$.

By \eqref{levi form} with
\begin{align}\label{general formula alt}
\Phi\left(z_1,\ldots, z_n,\overline{z_1},\ldots, \overline{z_{n}}\right):=\sum_{0<j<j+k\leq n}S_{j,k}x_jx_kx_n^{n-j-k},
\end{align}
the hypersurface $M=\left\{\left.z\in\mathbb{C}^{n+1}\,\right|\,\Re(z_0)=\Phi(z_1,\ldots,z_{n},\overline{z_1},\ldots, \overline{z_{n}})\right\}$ has Levi form represented by the matrix $H_{\mathcal{L}}$ equal to
\begin{align}\label{general levi form expl}
\bgroup
\arraycolsep=3pt
\left(
\begin{array}{ccccc;{2pt/2pt}c}
S_{1,1}x_n^{n-2}& S_{1,2}x_n^{n-3} & \cdots & S_{1,n-2}x_n& S_{1,n-1} & \sum_{j=1}^{n-2}jS_{n-1-j,1}x_{n-1-j}x_n^{j-1}\\
S_{2,1}x_n^{n-3} & & \iddots & S_{2,n-2} &0& \vdots \\
\vdots & \iddots& \iddots & \iddots   & \vdots & \sum_{j=1}^{2}jS_{3-j,n-3}x_{3-j}x_n^{j-1}\\
S_{n-2,1} x_n & S_{n-2,2}& \iddots & &  \vdots & \sum_{j=1}^{1}jS_{2-j,n-2}x_{2-j}x_n^{j-1}\\
S_{n-1,1} & 0& 0 &\cdots & 0 & 0 \\\hdashline[2pt/2pt]
\left( H_{\mathcal{L}}\right)_{n,1} &\cdots & \cdots & \left( H_{\mathcal{L}}\right)_{n,n-1} & 0 & \sum\limits_{j=1}^{n-3}\binom{j+1}{2}\left(\sum\limits_{ r+s=n-1-j}S_{r,s}x_rx_sx_n^{j-1}\right)
\end{array}
\right).
\egroup
\end{align}

Applying \eqref{bd normalization b} to simplify $H_{\mathcal{L}}$, its determinant simplifies to
\begin{align}\label{LForm determinant formula}
|H_{\mathcal{L}}|&=(-1)^n\left(H_{\mathcal{L}}\right)_{n,n}+\sum_{j=2}^{n-2}(-1)^{n+j+1}\left(H_{\mathcal{L}}\right)_{n-j,n}\left(H_{\mathcal{L}}\right)_{n,j},
\end{align}
which is equivalent to
\begin{align}\label{LForm determinant formula rearanged}
\left(H_{\mathcal{L}}\right)_{n,n}=\sum_{j=2}^{n-2}(-1)^{j}\left(H_{\mathcal{L}}\right)_{n-j,n}\left(H_{\mathcal{L}}\right)_{n,j}
\end{align}
because the Levi-degeneracy assumption implies $|H_{\mathcal{L}}|=0$.
For every nonnegative integer $m$, we can equate the coefficients of $x_n^m$ in \eqref{LForm determinant formula rearanged}, obtaining
\begin{align}\label{LForm determinant formula xnm coef}
\left(\sum\limits_{r+s=n-2-m}S_{r,s}x_rx_s\right)=\frac{2}{(m+2)!}\left.\frac{\partial^m}{\partial z_n^m}\right|_{\{z_n=0\}}\left(\sum_{j=2}^{n-2}(-1)^{j}\left(H_{\mathcal{L}}\right)_{n-j,n}\left(H_{\mathcal{L}}\right)_{n,j}\right).
\end{align}
The assumption that $|H_{\mathcal{L}}|=0$ is equivalent to \eqref{LForm determinant formula xnm coef} holding for all $m$.

Notice that, for each $r+s=n-2-m$, the coefficient of $x_{r}x_s$ on the left side of \eqref{LForm determinant formula xnm coef} is either $S_{r,s}$ or $2S_{r,s}$, whereas coefficients of $x_{r}x_s$ on the right side of \eqref{LForm determinant formula xnm coef} are functions of the parameters in $\{S_{j,k}\,|\,j+k>n-2-m\}$. This latter fact is easily observed from the general formula for $H_{\mathcal{L}}$ recorded above. 

Notice also that if $S_{j,k}=0$ for all $j+k= n-1$ then $\Phi$ does not define a uniformly $2$-nondegenerate hypersurface, that is, noting \eqref{bd normalization a},
\begin{align}\label{bd normalization c}
\sum_{j+k= n-1} S_{j,k}\neq0.
\end{align}
This can be seen in two ways, one by directly checking the criteria for $2$-nondegeracy at $z=0$, and the other by observing that \eqref{LForm determinant formula xnm coef} with $S_{j,k}=0$ for all $j+k= n-1$  implies  $S_{j,k}=0$ for all $j+k\leq n-1$, consequently implying that $\Phi$ defines a holomorphically degenerate hypersurface because it has no dependence on $z_n$.

In other words, by \eqref{LForm determinant formula xnm coef}, $S$ is fully determined by $\{S_{j,k}\,|\,j+k>n-2\}$. We normalized this set $\{S_{j,k}\,|\,j+k>n-2\}$ such that \eqref{bd normalization a}, \eqref{bd normalization b} and \eqref{bd normalization c} hold. The $2^p-1$ solutions, with $p=\left\lceil \tfrac{n-2}{2}\right\rceil$, referred to in the theorem's statement thus correspond to the $2^p-1$ possible parameter settings that can be assigned to $\{S_{j,k}\,|\,j+k>n-2\}$ satisfying \eqref{bd normalization a}, \eqref{bd normalization b} and \eqref{bd normalization c}. Each solution corresponds to one of the possible formulas for $S$ in \eqref{general coeficient matrix}.
\end{proof}

\section{Left-invariant structures}\label{Left-invariant structures on Lie groups}
The most symmetric among the homogeneous hypersurfaces of Theorem \ref{h structures theorem} turn out to be left-invariant structures on nilpotent Lie groups. In this section we describe these left-invariant structures along with their symmetry groups, obtaining preliminary results that will then be applied in Section \ref{Holomorphic classification}.

For the purposes of this section and Section \ref{Holomorphic classification} only, we switch notations, favoring \eqref{alternate notation}, and denote the coordinates as 
\[
(w,z_1,\ldots, z_{n-1},\zeta)\in\mathbb{C}^{n+1},\quad \mbox{ with }\quad \z=s+it 
\quad\mbox{ and }w=u+iv.
\]

We will analyze now the \emph{indecomposable model} $M_0$ given by 
\begin{align}\label{model}
M_0:=\left\{u=\sum_{0<j<j+k\leq n}x_jx_ks^{n-j-k}\right\}.
\end{align}
Consider the gradings
\begin{equation}\label{grading}
[\z]=1,\quad [w]=n,\quad [z_j]=j,\quad\quad\forall\,j=1,...,n-1.   \end{equation}
and 
\begin{equation}\label{grading0}
[\z]=0,\quad [w]=2,\quad [z_j]=1,\quad\quad\forall\,j=1,...,n-1. 
\end{equation}
The model $M_0$ has a weighted homogeneous defining equation with respect to both gradings \eqref{grading} and \eqref{grading0}, which has strong implications we will see for the structure of its algebra $\hol(M_0,0)$ of infinitesimal symmetries at $0$, namely, that the algebra admits a corresponding bi-grading.

This infinitesimal automorphism algebra $\hol(M_0,0)$ at the point $0\in M_0$ contains the bi-graded subalgebra
\begin{equation}\label{bi-graded}
\g=\g_{-n,-2}\oplus\left(\bigoplus_{j=1-n}^{-1}\g_{j,-1}\right)\oplus\g_{2-n,0}\oplus\g_{-1,0}\oplus\g_{0,0}  
\end{equation}
given by 
\begin{align}\label{gj}
&\g_{-n,-2}=\mbox{span}\left\{ Y_{-n} \right\},\quad \g_{-j,-1}=\mbox{span}\left\{ X_{-j},Y_{-j} \right\},\quad \forall\, j=1,...,n-1,\\
&\g_{2-n,0}=\mbox{span}\left\{V_{2-n}\right\},\quad\g_{-1,0}=\mbox{span}\left\{X'_{-1},Y'_{-1} \right\},\quad\mbox{ and }\quad \g_{0,0}=\mbox{span}\left\{ U_{0},V_{0} \right\},
\end{align}
where
\begin{align}\label{symmetry formulas a}
Y_{-n}=i\dw,\,Y'_{-1}=i\dzz,\,Y_{-j}=i\dz{j},\quad\forall\,j=1,...,n-1,
\end{align}
\begin{align}\label{symmetry formulas b}
X_{1-n}=\dz{n-1}+2z_1\dw,\,X_{-j}=\dz{j}-\z\dz{j+1}+2z_{n-j}\dw,\quad\forall\,j=1,...,n-2,
\end{align}
\begin{align}\label{symmetry formulas c}
X'_{-1}=2\dzz+\sum_{j=2}^{n-1}(2j-2-n)z_{j-1}\dz{j},\,V_{2-n}=iz_1^2\dw+iz_1\dz{n-1},
\end{align}
\begin{align}\label{symmetry formulas d}
U_0=\sum_{j=1}^{n-1}jz_j\dz{j}+\z\dzz+nw\dw,
\quad\mbox{ and }\quad
V_0=\sum_{j=1}^{n-1}z_j\dz{j}+2w\dw.
\end{align}

The bi-grading in \eqref{bi-graded} is naturally associated with the two gradings \eqref{grading} and \eqref{grading0} via the rule that vector fields are assigned weights such that
\begin{align}\left[\dzz\right]=-[\zeta],\quad \left[\dw\right]=-[w],\quad\left[\dz{j}\right]=-[z_j],\quad\quad\forall\,j=1,...,n-1,
\end{align}
yielding a different grading for each convention \eqref{grading} and \eqref{grading0}.
The grading \eqref{grading} is associated with the vector field $U_0$, as the coordinates in \eqref{grading} parameterize subspaces in the common eigenspaces of the weighted dilations generated by $U_0$ and the tuple $([\zeta],[w],[z_1],\ldots, [z_{n-1}])$ is a rescaling of the corresponding tuple of natural logarithms of eigenvalues for any one of those dilations. Similarly, \eqref{grading0} is associated with $V_0$.

The following lemma is obtained from straightforward, direct calculation, so we omit its proof here.
\begin{lemma}\label{model Lie group structure}$\,$

\begin{enumerate}
\item The vector fields
\[
e_j:=X_{-j}+iY_{-j},\quad e_{n-1+j}:=X_{-j}-iY_{-j}\quad\forall\,j=1,\ldots,n-1,
\]
\[
e_0:=4iY_{-n},\quad e_{2n-1}:=-X^\prime_{-1}-2iY^\prime_{-1}
\quad\mbox{ and }e_{2n}:=-X^\prime_{-1}+2iY^\prime_{-1},
\]
span a Lie algebra $\mathfrak{m}:=\mbox{span}_{\mathbb{C}}\{e_0,\ldots,e_{2n}\}$, with a graded decomposition $\mathfrak{m}=\mathfrak{m}_{-2}\oplus \mathfrak{m}_{-1}\oplus \mathfrak{m}_{0}$ given by $\mathfrak{m}_{-2}=\mbox{span}_{\mathbb{C}}\{e_{0}\}$, $\mathfrak{m}_{-1}=\mbox{span}_{\mathbb{C}}\{e_1,\ldots,e_{2n-2}\}$, and $\mathfrak{m}_{0}=\mbox{span}_{\mathbb{C}}\{e_{2n-1},e_{2n}\}$.
\item The subalgebra $\mathfrak{m}_{-}:=\mathfrak{m}_{-2}\oplus \mathfrak{m}_{-1}$, is a Heisenberg algebra with nonzero Lie brackets given by 
\[
[e_j,e_{2n-1-j}]=e_0\quad\forall\,j=1,\ldots, n-1.
\]
\item To describe the other Lie brackets in $\mathfrak{m}$, note that the adjoint representation of $\mathfrak{m}$ identifies $\mathfrak{m}_{0}$ with an abelian subalgebra in $\mathfrak{der}(\mathfrak{m}_{-})\cong\mathfrak{csp}(\mathfrak{m}_{-1})\subset \mathfrak{gl}(\mathfrak{m}_{-1})$ whose matrix representation with respect to the basis $(e_1,\ldots, e_{2n-2})$ is given by
\[
e_{2n-1}=
\left(
\begin{array}{c;{2pt/2pt}c}
 \Delta_{2-n,4-n,\ldots, n-2}  T_{n-1} & 2T_{n-1} \\\hdashline[2pt/2pt]
   0  & T_{n-1}\Delta_{2-n,4-n,\ldots, n-2}
\end{array}
\right)
\]
and
\[
e_{2n}=
\left(
\begin{array}{c;{2pt/2pt}c}
T_{n-1}\Delta_{2-n,4-n,\ldots, n-2} & 0 \\\hdashline[2pt/2pt]
   2T_{n-1}  &  \Delta_{2-n,4-n,\ldots, n-2}  T_{n-1}
\end{array}
\right)
\]
where $T_k$ denotes the $k\times k$ nilpotent Jordan matrix in \eqref{fundamental toeplitz} and $\Delta_{a_1,\ldots, a_k}$ denotes the diagonal matrix with values $a_1,\ldots, a_k$ along its main diagonal.
\item
The adjoint representation $\mbox{ad}_{V_0}$ of $V_0$ on $\mathfrak{m}$ has eigenspaces $\mathfrak{m}_{-2}$, $\mathfrak{m}_{-1}$, and $\mathfrak{m}_{0}$ with eigenvalue $j$ on $\mathfrak{m}_{j}$.
\end{enumerate}
\end{lemma}

The holomorphic fields $\{Y_{-n},X_{-j},Y_{-j},X^\prime_{-1},Y^\prime_{-1}\,|\,j=1,\ldots,n-1\}$ span a real form of the complex nilpotent Lie algebra $\mathfrak{m}$ in Lemma \ref{model Lie group structure}, so 
\[
\Re \mathfrak{m}:=\mathrm{span}_{\mathbb{R}}\left\{\Re(Y_{-n}),\Re(X_{-j}),\Re(Y_{-j}),\Re\left(X^\prime_{-1}\right),\Re\left(Y^\prime_{-1}\right)\,|\,j=1,\ldots,n-1\right\}
\]
is a $\dim(M_0)$-dimensional Lie algebra of nowhere vanishing infinitesimal symmetries on $M_0$. By equipping $M_0$ with the Lie group structure such that the fields in $\Re \mathfrak{m}$ are right-invariant, the CR structure on $M_0$ becomes left-invariant.

\begin{lemma}\label{grading lemma for model symmetry algebra}
The symmetry algebra $\mathbb{C}\otimes\hol(M_0,0)$ is $\mathbb{Z}$-graded through the decomposition into eigenspaces of $\mbox{ad}_{V_0}$ assigning respective eigenvalues as weights. In particular this grading is compatible with the secondary indices in the bi-grading \eqref{bi-graded}. The nontrivial negatively weighted components in this grading are 
\begin{align}\label{negative parts of symmetry alg}
\left(\mathbb{C}\otimes\hol(M_0,0)\right)_{-2}=\mathfrak{m}_{-2}
\quad\mbox{ and }\quad
\left(\mathbb{C}\otimes\hol(M_0,0)\right)_{-1}=\mathfrak{m}_{-1}
,
\end{align}
and if $n>4$ then the positively weighted components are all trivial.
\end{lemma}
\begin{proof}
Noting Lemma \ref{model Lie group structure}, this result follows readily from the main theorems in \cite{sykes2022maximal} and \cite{sykes2023geometry}.

In more detail, from the description of the subalgebra $\mathfrak{m}\subset \mathbb{C}\otimes\hol(M_0,0)$ in Lemma \ref{model Lie group structure}, it is immediately seen that $\mathfrak{m}$ satisfies the axioms of an \emph{abstract reduced modified symbol (ARMS)} given in \cite[Definition 6.1]{sykes2023geometry}, and hence the CR structure on $M_0$ is the flat structure with this ARMS, that is, a structure of the type described in \cite[part 3 of Theorem 6.2]{sykes2023geometry}. In particular, \cite[part 3 of Theorem 6.2]{sykes2023geometry} implies that $\mathbb{C}\otimes\hol(M_0,0)$ has the $\mathbb{Z}$-grading conferred by $\mbox{ad}_{V_0}$ with negative parts given by \eqref{negative parts of symmetry alg}. Vanishing of the positive weight components follows immediately from \cite[Theorem 3.7]{sykes2022maximal} and \cite[Theorem 6.2]{sykes2023geometry}. We require $n>4$ for this last conclusion so that \cite[Theorem 3.7]{sykes2022maximal} applies. Specifically, \cite[Theorem 3.7]{sykes2022maximal} does not apply if $n=3$ or $n=4$ because it is in exactly these cases that there structure's CR symbol is \emph{regular}.
\end{proof}

\begin{corollary}\label{model symmetries lemma corollary}
If $n>4$ then there exists a subspace $\mathfrak{I}\subset \{X\in\mathfrak{hol}(\mathbb{C}^{n+1},0)\,|\,[X,V_{0}]=0,X(0)=0\}$ such that
\[
\hol(M_0,0)=\mbox{span}_{\mathbb{R}}\{Y_{-n},X_{-j},Y_{-j},X^\prime_{-1},Y^\prime_{-1},U_0,V_0,V_{2-n}\,|\,j=1,\ldots,n-1\}\oplus \mathfrak{I}.
\]
\end{corollary}

We shall now prove
\begin{theorem}\label{modelalgebra}
For $n>4$, the symmetry algebra of the indecomposable model $M_0\subset\mathbb{C}^{n+1}$ (as in \eqref{model}) is exactly $\g$ (as in \eqref{bi-graded}), that is,
\[
\hol(M_0,0)=\mbox{span}_{\mathbb{R}}\{Y_{-n},X_{-j},Y_{-j},X^\prime_{-1},Y^\prime_{-1},U_0,V_0,V_{2-n}\,|\,j=1,\ldots,n-1\}.
\]
\end{theorem}

\begin{proof}
Given Corollary \ref{model symmetries lemma corollary}, it will suffice to describe all holomorphic vector fields $X$ on $\mathbb{C}^{n+1}$ that vanish at the origin, commute with $V_{0}$, and generate symmetries of $M_0$, so proceeding let $X$ be such a field.

Since $X$ commutes with $V_{0}$, it has the form
\begin{align}\label{Catlin multitype weight zero field}
 X=\left(a_{0,0}w+\sum_{j,k=1}^{n-1}a_{j,k}z_jz_k\right)\dw+\sum_{k=1}^{n-1}\left(\sum_{j=1}^{n-1}b_{j,k}z_j\right)\dz{k}+a\dzz,
\end{align}
where $a_{j,k}$, $b_{j,k}$, and $a$ are holomorphic functions of $\zeta$, and for simplicity we assume
\begin{align}\label{symmetry assumption}
a_{j,k}=a_{k,j}.
\end{align}
Indeed, \eqref{Catlin multitype weight zero field} clearly gives the general formula for vector fields that are weight zero with respect to \eqref{grading0}, which is an alternate characterization of the fields that commute with $V_0$. 



We start by considering the Lie bracket $$[Y_{-n},X]=ia_{0,0}\dw\in\mathfrak{m}_{-2}=\mbox{span}\{Y_{-n}\}$$
(we make use of Lemma \ref{grading lemma for model symmetry algebra}), so that $a_{0,0}$ is a real constant. This allows us to assume $a_{0,0}=0$, which can be done without losing generality by replacing $X$ with $X-\tfrac{a_{0,0}}{2}V_0$.

By Lemma \ref{grading lemma for model symmetry algebra}, $[Y_{-j},X]\in\mathfrak{m}_{-1}=\mbox{span}\{X_{-k},Y_{-k}\,|\,k=1,\ldots,n-1\}$ for all $j=1,\ldots,n-1$. 
Since
\begin{align}\label{YjX bracket}
[Y_{-j},X]=i\sum_{k=1}^{n-1}2a_{j,k}z_k\dw+i\sum_{k=1}^{n-1}b_{j,k}\dz{k},
\end{align}
it follows that $ia_{j,k}$ are all real constants and
\begin{align}\label{bjk coef equiv}
b_{j,1}=r_{j,1}+ a_{j,n-1},
\quad
b_{j,k}=r_{j,k}+ a_{j,n-k}- a_{j,n+1-k}\zeta, \quad\forall\, k=2,\ldots, n-1
\end{align}
for some real constants $r_{j,k}$.

We have
\begin{align}\label{XjX bracket a}
[X_{-j},X]&=\sum_{k=1}^{n-1}2(a_{j,k}-a_{j+1,k}\zeta)z_k\dw+\sum_{k=1}^{n-1}(b_{j,k}-b_{j+1,k}\zeta)\dz{k}\\
&\quad\quad+a\dz{j+1}-2\sum_{k=1}^{n-1}b_{k,n-j}z_k\dw\\
&=\left(\sum_{k=1}^{n-1}(b_{j,k}-b_{j+1,k}\zeta)\dz{k}\right)+a\dz{j+1}-2\sum_{k=1}^{n-1}r_{k,n-j}z_k\dw\quad\forall\, j<n-1,
\end{align}
 where the second equality holds by \eqref{symmetry assumption} and \eqref{bjk coef equiv}. Similarly,
\begin{align}\label{XjX bracket b}
[X_{1-n},X]
&=\sum_{k=1}^{n-1}b_{n-1,k}\dz{k}-2\sum_{k=1}^{n-1}r_{n-1,k}z_k\dw.
\end{align}
By Lemma \ref{grading lemma for model symmetry algebra}, $[X_{-j},X]\subset\mathfrak{m}_{-1}$, so, in particular, the coefficients in \eqref{XjX bracket a} should be at most linear in $\zeta$. Accordingly, from the $\dz{k}$ terms in \eqref{XjX bracket a} we get that $b_{j,k}$ is constant for all $j>1$ and that $a$ is linear in $\zeta$ (in fact $a(\zeta)=a'\zeta$ with $a'$ constant since $X(0)=0$). All this together with \eqref{bjk coef equiv} implies that all $b_{j,k}$ are constant and
\begin{align}\label{a coef are zero}
a_{j,k}=0 \quad\quad\forall\, 2\leq j,k\leq n-1\mbox{ with }j\neq n+1-k.
\end{align}
Examining \eqref{YjX bracket} again, since $[Y_{-j},X]\subset\mathfrak{m}_{-1}$, constancy of the $b_{j,k}$ implies 
\[
[Y_{-j},X]\subset\mbox{span}\{X_{1-n},Y_{-1},\ldots, Y_{1-n}\},
\]
which in turn implies $b_{j,n-1}$ is real for all $j>1$. Consequently, $a_{j,k}=0$ whenever $(j,k)\neq(1,1)$. Summarizing these conclusions, \eqref{Catlin multitype weight zero field} simplifies to
\begin{align}\label{Catlin multitype weight zero field a}
 X=\left(a_{1,1}z_1^2\right)\dw+a_{1,1}z_{1}\dz{n-1}+a\dzz+\sum_{j,k=1}^{n-1}r_{j,k}z_j\dz{k}.
\end{align}
where $a_{1,1}$ is purely imaginary. 

For simplicity, let us further replace $X$ with $X+ia_{1,1}V_{2-n}$, which allows us to assume $a_{1,1}=0$, that is, we will proceed assuming
\begin{align}\label{Catlin multitype weight zero field b}
 X=a'\zeta\dzz+\sum_{j,k=1}^{n-1}r_{j,k}z_j\dz{k},
\end{align}
for some real numbers $r_{j,k}$ and constant $a^\prime$. In addition, computing $$[Y_{-1}',X]=ia'\dzz\in\mathfrak m_{-1},$$ we see that $a'$ is real.

Applying these simplifications to \eqref{XjX bracket a} and noting $[X_{-j},X]\subset\mathfrak{m}_{-1}$, we get
\begin{align}\label{r coef constraint b}
r_{j,k}=r_{j+1,k+1}-\delta_{j,k}a^\prime\quad\quad\forall\, j<k<n-1,
\end{align}
where $\delta_{j,k}$ is the Kronecker delta. 

Evaluating $[X_{-k},[X_{-j},X]]$ with the simplifications of \eqref{Catlin multitype weight zero field b} yields
\[
[X_{-k},[X_{-j},X]]=-2\left(r_{k,n-j}-r_{k+1,n-j}\zeta+r_{j,n-k}-r_{j+1,n-k}\zeta\right)\dw 
\quad\quad\forall\, j,k<n-1,
\]
\[
[X_{1-n},[X_{-j},X]]=-2\left(r_{n-1,n-j}+r_{j,1}-r_{j+1,1}\zeta\right)\dw 
\quad\quad\forall\, j<n-1,
\]
and
\[
[X_{1-n},[X_{1-n},X]]=-2r_{n-1,1}\dw,
\]
and, since $[X_{-k},[X_{-j},X]]\subset\mbox{span}\{i\dw\}$ by Lemma \ref{grading lemma for model symmetry algebra}, these last three identities imply
\begin{align}\label{r coef constraint a}
r_{j,n-k}=-r_{k,n-j}\quad\quad\forall\, j,k.
\end{align}
Combining \eqref{r coef constraint b} and \eqref{r coef constraint a} we get
\[
r_{j,k}=
\begin{cases}
0 &\mbox{ if }j\neq k\\
\frac{2j-n}{2}a^\prime  &\mbox{ if }j= k
\end{cases}
\]
for all $1\leq j,k\leq n-1$. Applying this to simplify \eqref{Catlin multitype weight zero field b} further we get
\[
X=a^\prime U_0-\frac{na^\prime}{2}V_0\in \mathfrak{g},
\]
which completes the proof as we have shown an arbitrary field in $\hol(M_0,0)$ vanishing at $0$ belongs to $\mathfrak{g}$.
\end{proof}

\begin{remark}\label{modelalgebra remark}
Theorem \ref{modelalgebra} applies to models in $\mathbb{C}^N$ with $N>5$, excluding the $\mathbb{C}^4$ and $\mathbb{C}^5$ examples. These two lower-dimensional cases are exceptional as their symmetry algebras have more symmetries than just
\begin{align}\label{sym alg remark}
\mathfrak{g}=\mbox{span}\{Y_{-n},X_{-j},Y_{-j},X^\prime_{-1},Y^\prime_{-1},U_0,V_0,V_{2-n}\,|\,j=1,\ldots,n-1\}.
\end{align}
The exceptional behavior should be expected because the $\mathbb{C}^4$ and $\mathbb{C}^5$ structures considered in this paper have regular CR symbols, distinguishing them from all of the considered higher dimensional structures, and differences between the symmetry algebras of regular and nonregular symbol structures are well known \cite{porter2021absolute,sykes2023geometry}.

For the model $M_0$ in $\mathbb{C}^4$, its $16$-dimensional symmetry algebra is fully described in \cite[Section 6]{belweighted}. For the model $M_0$ in $\mathbb{C}^5$, its 14-dimensional symmetry algebra is spanned by \eqref{sym alg remark} and the fields
\begin{align}\label{c5 exceptional symmetry}
iz_1\dz{1}-iz_1\zeta\dz{2}+iz_3\dz{3}+i\zeta\dzz+2iz_1z_3\dw\quad\mbox{ and }\quad iz_1^2\dz{2}-2iz_1\dzz,
\end{align}
which can be verified by first checking that these formulas indeed define symmetries and then noting that $14$ is the upper bound for this structure's symmetry group dimension given by \cite[Theorem 6.2]{sykes2023geometry}; getting this latter bound requires computing an algebraic Tanaka prolongation, which incidentally is the same technique that we applied to derive the formula in \eqref{c5 exceptional symmetry}.
\end{remark}

\section{Classification of the infinitesimal automorphism algebras}\label{Classification of the infinitesimal automorphism algebras}

In this section, we provide a complete classification of the infinitesimal automorphism algebras $\hol(M,0)$ of all ({smooth}) hypersurfaces given by \eqref{main formula}, which through Lemma \ref{moving points to origin} (below) is reduced to the classification of such algebras at the origin for hypersurfaces given by \eqref{main formula} passing through the origin. These results include, in particular, the holomorphic classification of all the locally homogeneous hypersurfaces in \eqref{main formula}. We shall remind that, even in case $f$ is smooth, the Levi extension of CR functions from $M$ to its full neighborhood in $\CC{n+1}$ implies the analyticity of the infinitesimal automorphism algebra, i.e. every infinitesimal CR automorphism is the real part of an element of $\hol(M,0)$.

It turns out that in the non-homogeneous case the desired algebra $\hol(M,0)$ is a subalgebra in the model algebra $\g$, while in the homogeneous case one has to add to an (actually fixed) subalgebra in $\g$ a one-dimensional subalgebra. Our immediate goal is identifying such a one-dimensional subalgebra. As a bi-product, we shall obtain the holomorphic classification of the homogeneous hypersurfaces in \eqref{main formula}.  

Recall that $\hol(M,0)$ always contains the subalgebra $\mathfrak{f}\subset \mathfrak{g}$ spanned by 
\begin{equation}\label{falgebra}
Y_{-n}, \ldots,Y_{-1},Y_{-1}^{\prime},X_{1-n},\ldots,X_{-2},X^\prime_{-1},V_{2-n}
\end{equation}
(employing the notations of Section 6), and that the evaluation of $\mathfrak{f}$ at $0$ is a codimension $1$ subspace in $T_0M$. In view of that, the homogeneity of $M$ at $0$ amounts to the existence of a holomorphic vector field 
\[
X=\sum_{j=1}^{n-1}f_j(z,\z,w)\dz{j}+g(z,\z,w)\dzz+h(z,\z,w)\dw
\]
with $$\re f_1(0)\neq 0.$$ Subtracting further a linear combination of the vector fields in \eqref{falgebra}, we may assume 
\[
f_1(0)=1 
\quad\mbox{ and }\quad
f_2(0)=...=f_{n-1}(0)=g(0)=h(0)=0.
\]

Recall that, in view of the above Lemma \ref{normalizing quartic terms}, we may assume the vanishing of the derivatives of $f$ at $0$ up to order $3$, {meaning $f(0)=\cdots=f^{(3)}(0)=0$, and we denote this order of vanishing by $\mathrm{ord}_0\,f$}. We need next

 \begin{proposition}\label{nonhomog}
 If a hypersurface \eqref{main formula} is locally homogeneous at $0$, then either $f$ identically vanishes, or $\mathrm{ord}_0\,f=4$. 
 \end{proposition}
 \begin{proof}
 For the ``extra'' infinitesimal automorphism $X$ as above with $f_1(0)\neq 0,$  we have the tangency condition
 \begin{equation}\label{tang}
 \re h=\left.f_{x_1}\cdot\re f_1+\sum\nolimits_{j=1}^{n-1} P_{x_j}\cdot\re f_j+P_{s}\cdot\re g\right|_{w=f+P+iv},    
 \end{equation}
 where $f+P$ is the right hand side of \eqref{main formula}. 
 
 Now assume, to produce a contradiction, that the order is $m>4$. 
 We now  collect within \eqref{tang} all monomials in $x_1,y_1$ only of degree $m-1$. Then we see that such terms come from three sources: the differentiated $x_1^m$ term in $f$ combined with $f_1(0)$, the polynomial $2x_1$ multiplied by the real part of the $z_1^{m-2}$ term within $f_{n-1}$, and finally the real part  of the $z_1^{m-1}$ term within $h$. This means that for some complex numbers $A,C$ and a real number $B\neq 0$ it holds that:
 $$\re(Az_1^{m-1})=Bx_1^{m-1}+x_1\re(Cz_1^{m-2}).$$
 Applying the univariate (in $z_1$) Laplacian to this equation twice yields the contradiction $x_1^{m-5}=0$.
 \end{proof}
 In view of Proposition \ref{nonhomog}, we may assume that {either $f=0$ or} $f^{(4)}(0)=1$ for a homogeneous hypersurface and further $f^{(5)}(0)=\sigma,$ where $\sigma$ is either $0$ or $1$ (which is accomplished by applying scalings $U_0$, $V_0$ in the notations of Section 6). 
 
 We shall now refer again to the tangency condition \eqref{tang}. Using the multitype weights (corresponding to the vector field $V_0$), we collect in \eqref{tang} terms of weights $k$ for $k=0,1,...$ (replacing for this purpose $f$ by its Taylor series at $0$). We recall the standard ``homological'' argument: the respective identities in each weight $k$ have the form 
 $$L(f_1^{k-1},...,f_{n-1}^{k-1},g^{k-2},h^k)=\cdots,$$
 where $L$ is the homological operator 
 $$L(f_{1},\ldots, f_{n-1},g,h):=\re h-\sum_{j=1}^{n-1}P_{x_{j}}\re f_j-P_s\re g\Bigr|_{w=P+iv},$$
 and dots stand for a linear expression in lower weight terms, if any (the upper indices here stand for the respective weighted homogeneous components). Then one successively determines each collection $(f_1^{k-1},...,f_{n-1}^{k-1},g^{k-2},h^k)$ by knowing the previous ones, and such a determination is unique modulo the {\em kernel} of the operator, while the latter precisely corresponds to the infinitesimal automorphism algebra $\g$ of the model.   The vanishing of $f(0)=f^\prime(0)=f^{\prime\prime}(0)=f^{(3)}(0)$ implies now that the  identities under discussion in the weights $0,1,2$ are {\em identical} to that in the case of the model $M_0$, which is why the collections $(f_1^{k-1},...,f_{n-1}^{k-1},g^{k-2},h^k)$ corresponding to $k=0,1,2$ coincide with that coming from the model algebra $\g$. Now, subtracting if necessary a constant multiple of the vector field $V_{2-n}$, we conclude that
 \begin{equation}\label{expandX}
X=X_{-1}+pU_0+qV_0+O(1),     
 \end{equation}
 where $X_{-1}$ comes from the model algebra $\g$ as in \eqref{symmetry formulas c}, $p,q$ are real numbers, and $O(1)$ stands for a vector field whose expansion at $0$ contains only terms of weights greater than or equal to $1$.
 
 We next deal with the stability subalgebra $\aut(M,0)$. 
 
\begin{proposition}\label{genalgebra}
 For a hypersurface \eqref{main formula} with $f\not\equiv 0$ and $n>4$, either of the following holds:

 \smallskip

 (I) The function $f(x_1)$ is a monomial $ax_1^m,\,m
 \geq 4$ and then $\aut(M,0)$ is spanned by $V_{2-n}$ and the vector field 
 \begin{equation}\label{mfield}
 U_0^m:=\sum_{j=1}^{n-1}\left(\frac{n-m}{n-2}+\frac{m-2}{n-2}j\right)z_j\dz{j}+\frac{m-2}{n-2}\z\dzz+mw\dw,    
 \end{equation}
 or
  \smallskip
  
(II)  the function $f(x_1)$ is not a monomial, and $\aut(M,0)$ is spanned by $V_{2-n}.$

 \end{proposition}

 \begin{proof}
 Taking a vector field $X\in\aut(M,0)$ and arguing as in the proof of the expansion \eqref{expandX}, we obtain an analogous expansion
 $$X=pU_0+qV_0+rV_{2-n}+Y,$$
 where $Y$ is a vector field whose expansion at $0$ contains only terms of weights $\geq m-3$ (recall that $X(0)=0$). However, observe that $Y$ actually vanishes identically. Indeed, it vanishes identically in the case of the model $M_0$, as follows from the description of $\g$, while the above homological argument implies the bound 
 \begin{equation}\label{bd}
 \mbox{dim}\,\mathcal V_k(M) \leq    \mbox{dim}\,\mathcal V_k(M_0)\quad \forall k\in\mathbb Z,  
 \end{equation}
 where $\mathcal V_k(M)$ denotes the linear subspace in $\hol (M,0)$ spanned by vector fields whose expansion at $0$ do not contain terms of weight less than $k$ (compare with Remark 12 in \cite{belweighted}). This implies $\mbox{dim}\,\mathcal V_{m-3}(M)=0$ and hence $Y\equiv 0$ (recall that $Y$ is a analytic at $0$, unlike $f$ which can be merely smooth).
 
We conclude  that $X-rV_{2-n}$ is the scaling field $p U_0 +q V_0$, so that $p U_0 +q V_0\in\aut(M,0)$. It is straightforward to verify that for $(p,q)\neq (0,0)$ $M$ is stable under dilations in $p U_0 +q V_0$ if and only if $f$ is a {\em monomial} and $p U_0 +q V_0$ is proportional to the above $U_0^m$ (the latter is true even in case $f$ is merely smooth). This finally implies the assertion of the proposition.
 \end{proof}

 Returning to the homogeneous case, the expansion \eqref{expandX} appears to be very useful for determining the structure of the additional vector field $X$ as above. Let us further expand
 $$X=X_{-1}+pU_0+qV_0+rV_{2-n}+X_1+X_2+
 \cdots,$$
 where $X_j$ is a vector field of weight $j$. 
 Now, since all the commutators of $X$ with $Y_j$ have to be spanned by vector fields from $\hol(M,0)$, we first consider within $[Y_n,X]$ the $(-1)$ weighted component. In view of the above, the latter has to be spanned by $\{X_{-j},Y_{-j}\,|\,1\leq j\leq n-1\}$.  On the other hand, this component equals $[Y_{-n},X_1]$. This implies  $$(F_1)_w=c_1, \, (F_2)_w=c_2+C_2\cdot\zeta,  ...(F_{n-1})_w=c_{n-1}+C_{n-1}\cdot\zeta,\,G_w=0, \,  H_w=A(z)$$
 for appropriate constants $c_j,C_j$ and a real linear function $A(z)$,  where $$X_1=\sum_{j=1}^{n-1} F_j(z,\z,w)\dz{j}+G(z,\z,w)\dzz+H(z,\z,w)\dw.$$
 We proceed with a similar consideration for $[Y_{-j},X],\,j=1,...,n-1$, which all have to be linear combinations of elements of the algebra $\hol(M,0)$. Picking the weight $0$ component of the bracket, we conclude that $[Y_{-j},X_1]$ is a linear combination of $U_0,V_0,V_{2-n},Y_{-1}',X_{-1}'$. This allows, in view of the already established conditions, the following freedom for the components of $X_1$:
 \[
 H=2tz_1^3,\,F_1=c_1w+iQ_1(z_1),\,F_j=(c_j+C_j\z) w+iQ_j(z_{j-1},z_j)\quad \forall\,j=2,...,n-2,
 \]
 $$F_{n-1}=(c_{n-1}+C_{n-1}\z)+3tz_1^2+iQ_{n-1}(z_{n-2},z_{n-1}),\,G=l_1(\z)z_1+...+l_{n-1}(\z)z_{n-1}$$
 for an appropriate constant $t\in\RR{}$, affine polynomials $l_j$, and real homogeneous quadratic polynomials $Q_j$ (the linear function $A(z)$ vanishes in view of its reality). 
 
 Considering now the Lie bracket of $X$ with $Y_{-1}'$, we pick the weight $1$ component within it. This component equals $[Y_{-1}',X_1]$ and is at the same time a constant multiple of $X_1$. We immediately conclude that this Lie bracket vanishes, all the $l_j$ are constant, and all the $C_j$ vanish as well.  Coming back to the Lie brackets $[Y_{-j},X_1]$, we consider the $\dzz$ component and see (from the constancy of $l_j$) that the expansion of none of these Lie brackets contains $U_0$. At the same time, for each $j=1,...,n-1$ we are able to find some $k$ such that $F_k$ is independent of $z_j$ (as follows from the above description). This implies that the expansion of none of the Lie brackets $[Y_{-j},X_1]$ contains $V_0$ (and for a similar reason none contain $X_{-1}'$). It follows from here that all the $Q_j$ vanish. Thus,
 $$X_1=c_1w\dz{1}+...+c_{n-2}w\dz{n-2}+(c_{n-1}w+3tz_1^2)\dz{n-1}+(l_1z_1+...+l_{n-1}z_{n-1})\dzz+2tz_1^3\dw.$$
 We now consider the commutator $[X_{1-n},X]$, pick the $0$ weighted component (which equals $[X_{1-n},X_1]$), and obtain
 $$2z_1c_1\dz{1}+2z_1c_2\dz{2}+...+2z_1c_{n-1}\dz{n-1}+l_{n-1}\dzz-2c_1w\dw.$$
 But  $[X_{1-n},X]$ has to be in the span of $U_0,V_0,V_{2-n},Y_{-1}',X_{-1}'$, as above, so $c_1,..,c_{n-1}$ must all vanish.
 
 We finally take the commutators $[X_{-j},X],\,j=2,...,n-2$,  again pick the $0$ weighted component (which equals $[X_{-j},X_1]$), and obtain:
 $$(l_1z_1+...+l_{n-1}z_{n-1})\dz{j+1}+(l_j-l_{j+1}\z)\dzz.$$
 But  each $[X_{-j},X]$ has to be in the span of $U_0,V_0,V_{2-n},Y_{-1}',X_{-1}'$, as above, and thus all the $l_j$ vanish.

This finally gives
 \begin{equation}\label{X1}
 X_1=3tz_1^2\dz{n-1}+2tz_1^3\dw,\quad t\in\RR{}.
 \end{equation}
Observe that the sum of all the negatively graded terms of the commutator $[Y_{-1},X]$ is spanned by the vector fields $Y_{-1}$ and $X_{1-n}$. This implies that the expansion of  $[Y_{-1},X]$ in the Lie algebra $\hol(M,0)$ does {\em not} contain $X$. 
 
 We claim now that one has $X_2=X_3=...=0$. For the component $X_2$, we consider the brackets $[Y_{-j},X]$ and their weighted component of the weight $1$. The latter coincide with respectively $[Y_{-j},X_2]$ and have to be a constant multiple of $X_1$. The latter is possible only if, up to a nonzero real multiple,
 \begin{equation}\label{X2}
 X_2=itz_1^3\dz{n-1}+\left(\frac{i}{2}tz_1^4+A(\z)w^2\right)\dw.
 \end{equation}
 The consideration of the bracket $[Y_{-n},X]$ easily gives $A=0$. And the above observation on the expansion of $[Y_{-1},X]$ gives lastly $X_2=0$. 
 
 We use the conclusion  $X_2=0$ to prove that all the components of $X_3$ are independent of $z_j$ (by commuting with the respective $Y_{-j}$), and then the commutator with $Y_{-n}$ to conclude that $X_3=0$. After that the conclusion $X_4=X_5=...=0$ follows by induction (after commuting $X$ with all the $Y_{-j}$).
 
Summarizing these observations, we have shown 
 
 \smallskip
 
 {\em \noindent for a homogeneous at $0$ hypersurface \eqref{main formula} the above extra vector field $X$ must have the form:
 \begin{equation}\label{extra}
X=X_{-1}+pU_0+qV_0+rV_{2-n}+X_1,     
 \end{equation}
 where $X_1$ is as in \eqref{X1}}.
 
 We consider now a hypersurface $M$ with 
 \begin{equation}\label{normalf}
 f(x_1)=x_1^4+\sigma x_1^5+\sum_{j=6}^\infty a_jx^j,\quad a_j\in\RR{},
 \end{equation}
 normalized as above, with its infinitesimal automorphism \eqref{extra}, and study its  tangency condition \eqref{tang}. By substracting $rV_{2-n}$ from $X$, we may always assume that $r=0$, for convenience. As discussed above, the tangency condition is satisfied in the weights $1,2$.  In weight $3$, the condition gives
  \begin{equation}\label{weight3}
2t\re(z_1^3)=4x_1^3+6tx_1\re(z_1^2),
 \end{equation}
 which is equivalent to requiring $t=-1$.
 
 Finally,
 for all the weights greater than $3$, considered {\em at once}, we get:
 \begin{equation}\label{weights4}
(np+2q)f(x_1) = \bigl(1+(p+q)x_1\bigr)f'(x_1)-4x_1^3.      
 \end{equation}
 Particularly in the weight $4$ we have the compatibility condition 
 $$np+2q=5\sigma+4(p+q).$$
 We now have to describe all solutions for the family of 1st order ODEs \eqref{weights4} satisfying \eqref{normalf} (and hence the compatibility conditions). The respective ODEs are all linear inhomogeneous and can be solved by elementary methods explicitely, that is why we omit the calculation of solutions and just present the final outcome. It will be convenient to not insist on the normalization \eqref{normalf} at this point and require merely $f(0)=f'(0)=0$ (recall, by Lemma \ref{normalizing quartic terms}, terms of degrees $2$ and $3$ in $x_1$ can be removed from $f$ without changing the form \eqref{main formula} and the higher order terms within $f$). Acting on the set of solutions as above by the transformation group generated by:
 \begin{equation}\label{Ggroup}
     f\mapsto f+ax_1^2+bx_1^3,\quad f\mapsto c_1f(c_2x_1),\quad a,b,c_1,c_2\in\RR{},\,c_1,c_2\neq 0,
     \end{equation}
 we obtain the following final classification (also summarized in the introduction's Theorem \ref{h structures theorem}).
 
 {\bf Type I}. Here $f(x)=(x+1)^a-ax-1 $ for some $a\in \mathbb{R},\,a\neq\{0,1,2,3\}$. In this case we compute the extra vector field $X$ as
\begin{align}
V_f&:=\left(aw+a(a-1)z_1+2z_{n-1}\right)\dw+\frac{\partial}{\partial z_{1}}-\z\frac{\partial}{\partial z_{2}}+\frac{a-2}{n-2}\z\dzz+\\
&\quad\quad+\sum_{j=1}^{n-1}\left(\frac{n-a}{n-2}+\frac{a-2}{n-2}j\right)z_j\frac{\partial}{\partial z_{j}}.
\end{align}

 {\bf Type II}. Here $f(x)=e^x-x-1$ and we compute the extra vector field $X$ as
\begin{align}
V_f&:=\left(w+z_1+2z_{n-1}\right)\dw+\frac{\partial}{\partial z_{1}}-\z\frac{\partial}{\partial z_{2}}+\frac{1}{n-2}\z\dzz+\sum_{j=1}^{n-1} \frac{j-1}{n-2}z_j\frac{\partial}{\partial z_{j}}.
\end{align}

 {\bf Type III}. Here $f(x)=\ln(x+1)-x$ and we compute the extra vector field $X$ as
\begin{align}
V_f&:=\left(2z_{n-1}-z_1\right)\dw+\frac{\partial}{\partial z_{1}}-\z\frac{\partial}{\partial z_{2}}-\frac{2}{n-2}\z\dzz+\sum_{j=1}^{n-1} \frac{n-2j}{n-2}z_j\frac{\partial}{\partial z_{j}}.
\end{align}

 {\bf Type IV}. Here $f(x)=(x+1)\ln(x+1)-x$ and we compute the extra vector field $X$ as
\begin{align}
V_f&:=\left(w+z_1+2z_{n-1}\right)\dw+\frac{\partial}{\partial z_{1}}-\z\frac{\partial}{\partial z_{2}}-\frac{1}{n-2}\z\dzz+\sum_{j=1}^{n-1} \frac{n-j-1}{n-2}z_j\frac{\partial}{\partial z_{j}}.
\end{align}

 {\bf Type V}. Here $f(x)=2(x+1)^2\ln(x+1)-(x+1)^2+1$ and we compute the extra vector field $X$ as
\begin{align}
V_f&:=\left(2w+4z_1+2z_{n-1}\right)\dw+\frac{\partial}{\partial z_{1}}-\z\frac{\partial}{\partial z_{2}}+(z_{n-1}-z_1)\frac{\partial}{\partial z_{n-1}}+\sum_{j=1}^{n-2} z_j\frac{\partial}{\partial z_{j}}.
\end{align}

 {\bf Type VI}. Here $f(x)=6(x+1)^3\ln(x+1)-2(x+1)^3+2$ and we compute the extra vector field $X$ as
\begin{align}
V_f&:=\left(3w+18z_1-3z_1^3 + 2z_{n-1}\right)\dw+\frac{\partial}{\partial z_{1}}-\z\frac{\partial}{\partial z_{2}}-\left(9z_1-\frac{9}{2}z_1^2\right)\frac{\partial}{\partial z_{n-1}}\\
&\quad\quad+\frac{1}{n-2}\z\dzz+\sum_{j=1}^{n-1} \frac{n-3+j}{n-2}z_j\frac{\partial}{\partial z_{j}}.
\end{align}
Note that, in turn, for the homogeneous case, excluding $f\equiv 0$ and $f\equiv x_1^4$, we always have $\sigma=1$ in \eqref{normalf}.

Having found $\mathfrak{aut}(M,0)$ for all $M$ given by \eqref{main formula}, descriptions of $\mathfrak{aut}(M,z^*)$ for arbitrary $z^*\in M$ are given by the following Lemma.
\begin{lemma}\label{moving points to origin}
    For a hypersurface $M_f$ of the form in \eqref{main formula} and a point $z^*=(z_0^*,\ldots,z_n^*)\in M_f\subset \mathbb{C}^{n+1}$, there exists a Lie algebra isomorphism identifying
    \[
\mathfrak{aut}(M_{f(x)},z^*)\cong\mathfrak{aut}(M_{f(x-x_1^*)-f(-x_1^*)},0),
    \]
where $x_1^*$ denotes $\Re(z_1^*)$.
\end{lemma}
\begin{proof}
We will consider how $M_f$ transforms under the symmetries of the indecomposable model $M_0$ (i.e., the hypersurface given by \eqref{main formula} with $f=0$), considering, in particular, flows of the infinitesimal symmetries $X_{-1},\ldots, X_{1-n}$, $Y_{-1},\ldots, Y_{-n}$, $X_{-1}^\prime$, and $Y_{-1}^\prime$ given by \eqref{symmetry formulas a},\eqref{symmetry formulas b}, and \eqref{symmetry formulas c}. 

Fixing some notation, using coordinates $z=(z_0,\ldots, z_n)$ of $\mathbb{C}^{n+1}$, let $\pi_j(z):=z_j$ denote the coordinate projections, and let $\mathrm{Fl}^X_t:\mathbb{C}^{n+1}\to \mathbb{C}^{n+1}$  denote the flow of the real part of a holomorphic vector field $X\in \mathfrak{hol}(\mathbb{C}^{n+1})$, that is
\begin{align}\label{flow formula}
\left.\tfrac{d}{dt}\mathrm{Fl}^{X}_{t}(z)\right|_{t=0}=\Re(X)(z)\quad\quad\forall\, z\in \mathbb{C}^{n+1}.
\end{align}
For all $t\in\mathbb{R}$ and $j\in\{1,\ldots, n-2\}$, the index $r$ component of $\mathrm{Fl}^{X_{-j}}_{t}$ is given by
\[
\pi_r\circ\mathrm{Fl}^{X_{-j}}_{t}(z)=z_r+\delta_{r,j}t-\delta_{r,j+1}tz_n+\delta_{r,0}(2tz_{n-j}+\alpha_{n,j,t}),
\]
where $\delta_{r,j}$ is the Kronecker delta and 
\[
\alpha_{n,j,t}:=
\begin{cases}
    t^2&\mbox{ if }n=2j\\
    -t^2z_n&\mbox{ if }n=2j+1\\
    0 & \mbox{ otherwise.}
\end{cases}
\]
Additionally we have
\[
\pi_r\circ\mathrm{Fl}^{Y_{-1}^\prime}_{t}(z)=z_r+\delta_{r,n}it,
\quad
\pi_r\circ\mathrm{Fl}^{Y_{-j}}_{t}(z)=z_r+\delta_{r,j}it
\quad\quad\forall\, j\in\{1,\ldots, n-1\},
\]
\[
\pi_r\circ\mathrm{Fl}^{Y_{-n}}_{t}(z)=z_r+\delta_{r,0}it,
\quad
\pi_r\circ\mathrm{Fl}^{X_{1-n}}_{t}(z)=z_{r}+\delta_{r,n-1}t-\delta_{r,0}2tz_{1},
\]
and
\[
\pi_r\circ\mathrm{Fl}^{X_{-1}^\prime}_{t}(z)=z_{r}+\delta_{r,n}2t+\beta_{n,r}(t),
\]
where $\beta_{n,0}(t),\ldots,\beta_{n,n-1}(t)$ are functionals on $\mathbb{C}^{n+1}$ parameterized by $t$ defined recursively as $\beta_{n,0}(t)=\beta_{n,1}(t)=\beta_{n,n}(t)=0$, $\beta_{n,2}(t)=(2-n)z_1t$ and
\[
\frac{\partial}{\partial t}\beta_{n,j}=(2j-2-n)(z_{j-1}+\beta_{n,j-1})
\quad\mbox{ and }\quad
\beta_{n,j}(0)=0\quad\quad\forall j\in\{3,\ldots, n-1\}.
\]
It is easily checked that these formulas all satisfy \eqref{flow formula}.

For any $j\in\{1,\ldots,n-1\} $ and $v\in \mathbb{C}^N$ with $\pi_r(v)=0$ for all $r=1,\ldots, j-1$, we have
\[
\pi_r\circ\mathrm{Fl}_{-y_1}^{Y_{-j}}\circ\mathrm{Fl}_{-x_1}^{X_{-j}}(v)=0
\quad\quad\forall\, r=1,\ldots, j,
\]
whereas if $\pi_r(v)=0$ for all $r=1,\ldots, n-1$ then
\[
\pi_r\circ\mathrm{Fl}_{-y_n}^{Y_{-1}^\prime}\circ\mathrm{Fl}_{-\frac{x_n}{2}}^{X_{-1}^\prime}(v)=0
\quad\quad\forall\, r=1,\ldots, n.
\]

Hence we can apply such transformations in sequence to the point $z^*\in \mathbb{C}^N$ to transform it to a point of the form $(w,0,\ldots,0)$. Specifically, there exist vectors $v,u\in\mathbb{C}^n$ such that 
\[
\Psi:=\mathrm{Fl}_{u_n}^{Y_{-1}^\prime}\circ\mathrm{Fl}_{v_n}^{X_{-1}^{\prime}}\circ\mathrm{Fl}_{-u_{n-1}}^{Y_{1-n}}\circ\mathrm{Fl}_{-v_{n-1}}^{X_{1-n}}\circ\cdots\circ\mathrm{Fl}_{u_2}^{Y_{-2}}\circ\mathrm{Fl}_{v_2}^{X_{-2}}\circ\mathrm{Fl}_{u_1}^{Y_{-1}}\circ\mathrm{Fl}_{v_1}^{X_{-1}}
\]
satisfies $\pi_1\circ\Psi(z)=z_1-z_1^*$, and $\Psi(z^*)=(\tilde w,0,\ldots, 0)$ for some number $\tilde w$.  In particular, using the property $F\circ \Psi^{-1}=F$ for 
\[
F(z)=x_0-\sum_{0<j<j+k\leq n}x_jx_kx_n^{n-j-k},
\]
which holds because $\Psi$ is constructed from flows of vector fields annihilating $F$, we have $\Psi(z^*)=(f(x_1^*),0,\ldots, 0)$. Thus,
\begin{align}\label{isotropy isom}
\Psi(M_{f(x)})=M_{{f(x+x_1^*)}}
\quad\mbox{ and }\quad
\Psi(z^*)=(f(x_1^*),0,\ldots,0).
\end{align}
Hence, letting $\Phi$ denote the translation $\Phi(z)=(z_0-f(x_1^*),z_1,\ldots, z_n)$, the induced map $(\Phi\circ\Psi)_*:\mathfrak{aut}(M_{f(x)},z^*)\to\mathfrak{aut}(M_{f(x+x_1^*)-f(x_1^*)},0)$ is an isomorphism.
\end{proof}

We summarize the outcomes of this section in the following classification theorem.

\begin{theorem}\label{thealgebras}
For a {smooth} hypersurface $M$ as in \eqref{main formula} and a point $z^*\in M$, there is a holomorphic change of coordinates $\Psi$ (given in the proofs of Lemmas \ref{normalizing quartic terms} and \ref{moving points to origin}) such that $\Psi(z^*)=0$ and $\Psi$ transforms $M$ into a {smooth} hypersurface as in \eqref{main formula} normalized to satisfy $f(0)=f^\prime(0)=f^{\prime\prime}(0)=f^{(3)}(0)=0$. After such normalization, its infinitesimal automorphism algebra $\hol(M,0)$ is as follows.

\smallskip 

\noindent (i) If $f\equiv 0$, then $\hol(M,0)$ is the model algebra $\g$, as in Theorem \ref{modelalgebra}.

\smallskip 

\noindent (ii) If $f$ is a monomial $ax_1^4$ with $a\neq 0,$ then $M$ is locally homogeneous at $0$, the algebra $\hol(M,0)$ is spanned by the algebra $\mathfrak f$, as in \eqref{falgebra}, the vector field $U_0^4$, as in Proposition \ref{genalgebra}, and the vector field
\[
\widehat X_{-1}:=
\left(2z_{n-1}-2az_1^3\right)\frac{\partial}{\partial z_{0}}+\frac{\partial}{\partial z_{1}}-z_n\frac{\partial}{\partial z_{2}}-3az_1^2\frac{\partial}{\partial z_{n-1}}.
\]

\smallskip 

\noindent (iii) If $f$ is a monomial $ax_1^m$ with $m\geq 5$ and $a\neq 0$, then $\hol(M,0)$ is spanned by the algebra $\mathfrak f$, as in \eqref{falgebra}, and the vector field $U_0^m$, as in Proposition \ref{genalgebra}.

\smallskip 

\noindent (iv) If $M$ is locally homogeneous at $0$, and the respective $f$ is not identically $0$ and not a monomial $ax_1^4$, then $f$ is under the group action of \eqref{Ggroup}, equivalent to one of the functions I - VI above (with $a\neq 4$), while $M$ is  $\phi^{-1}$-biholomorphic to the respective $M_f$, where $\phi$ is a composition of the flows of $U_0,V_0$ and the map \eqref{remove}. The algebra $\hol(M,0)$ is spanned by the algebra $\mathfrak f$, as in \eqref{falgebra}, and the $\phi$-pushforward of the respective vector field $V_f$ (which necessarilly has the form \eqref{extra}). 

\smallskip 

\noindent (v) If $M$ is not as in (i)-(iv), then $\hol(M,0)=\mathfrak f$.

\smallskip

In particular, any locally homogeneous at $0$ hypersurface \eqref{main formula} satisfies either (i), (ii), or (iv), and all the respective homogeneous models are biholomorphically inequivalent (including type I hypersurfaces with different $a$).

\end{theorem}

The last assertion (on the holomorphic inequivalence of the models I-VI) follows from Theorem \ref{Holomorphic classification theorem}, proved in the next section.

\section{Holomorphic classification}\label{Holomorphic classification}
In this section we prove Theorem \ref{Holomorphic classification theorem} and its Corollary \ref{Holomorphic classification corollary}, which solves the local equivalence problem for {smooth} hypersurfaces of Theorem \ref{main theorem} by deriving a criteria for when one such hypersurface can be obtained as the image of a biholomorphism restricted to another. {Notice that Corollary \ref{Holomorphic classification corollary} indeed follows immediately from Theorem \ref{Holomorphic classification theorem} given properties of the map $\Phi\circ\Psi$ constructed }in the proof of Lemma \ref{moving points to origin}, namely for a point $z^*\in \mathbb{C}^{n+1}$ this bi-holomorphic map is constructed such that
\begin{align}\label{isotropy isom a}
\Phi\circ\Psi(M_{f(x)})=M_{{f(x+x_1^*)}-f(x_1^*)},
\quad
\pi_1\circ\Phi\circ\Psi(z)=z_1-z_1^*,
\quad\mbox{ and }\quad
\Phi\circ\Psi(z^*)=0.
\end{align}

For this section only we switch notation labeling $z=(z_1,\ldots, z_{n-1})$, so that coordinates in $\mathbb{C}^{n+1}$ can be denoted efficiently by $(w,z,\zeta)$.

Let us now prove Theorem \ref{Holomorphic classification theorem}.

\begin{proof}
Applying Lemma \ref{normalizing quartic terms}, we will assume the nonzero terms in the Maclaurin expansions of $f$ and $f^*$ are of order greater than $3$.

Suppose that there exists a biholomorphism
\begin{align}\label{transformation of hypersurfaces}
\Phi(w,z,\zeta):=(w^*,z^*,\zeta^*)=(w+h,z_1+f_{1},\ldots, z_{n-1}+f_{n-1},\zeta+g)
\end{align}
of $\mathbb{C}^{n+1}$ satisfying $\Phi(0)=0$ and $\Phi(M_f\cap U)=M_{f_*}\cap \Phi(U)$ for some neighborhood $U\subset \mathbb{C}^{n+1}$ containing $0$, where $h, f_1,\ldots,f_{n-1},  g$ are holomorphic functions of $w,z_1,\ldots, z_{n-1},\zeta$ vanishing at $0$, and $z$ and $z^*$ abbreviate $z_1,\ldots, z_{n-1}$ and $z_1^*,\ldots, z_{n-1}^*$ respectively. Letting $h^k$, $f_{j}^k$, and $g^k$ denote the sum of the weight $k$ terms in $h$, $f_j$, and $g$ respectively with respect to the weights in \eqref{grading0}, we have
\[
w^*=w+h^0+h^1+\cdots,\quad
z_j^*=z_j+f_{j}^0+f_{j}^1+\cdots,\quad\mbox{and}\quad
\zeta^*=\zeta+g^0+g^1+\cdots.
\]
Letting $P$ be the defining function given by \eqref{main formula} with $f=0$, set $F:=P+f$ and $F^*=P+f^*$, that is, $u=F$ and $u^*=F^*$ are defining equations for $M_f$ and $M_{f^*}$ respectively.

Since $\Phi(M_f)=M_{f_*}$ near $0$, we have
\begin{align}\label{transformation of defining equation}
\quad P(z,\zeta)+f\left(\Re(z_1)\right)+\Re\left(h(F+iv,z,\zeta)\right)=P(z^*,\zeta^*)+f^*\left(\Re(z_1^*)\right)
\,\,\forall\,(w,z,\zeta)\in M_f
\end{align}
near $0$, because the left side of \eqref{transformation of defining equation} is equal to $\Re(w^*)$ whenever $(w,z_1,\ldots, z_{n-1},\zeta)\in M_f$. 

One sees quickly from \eqref{transformation of defining equation} that 
\begin{align}\label{low degree transformation terms}
h^0=h^1=f_{1}^0=\cdots=f_{n-1}^0=0.
\end{align}
Indeed, we get \eqref{low degree transformation terms} from the following observations. First, suppose $f_{1}^0,\cdots,f_{n-1}^0$ are not all zero, and pick $j$ such that $f_{j}^0$ has the lowest order of vanishing (i.e., a lowest degree nonzero term in $\zeta$). Letting this order be $I$, we have
\[
f_{j}^0(\zeta)=c\zeta^I+O(\zeta^{I+1})
\]
for some number $I\in\mathbb{N}$ and $c\in \mathbb{C}\setminus\{0\}$. The right side of \eqref{transformation of defining equation} has a mixed term that is a nonzero multiple of $z_{n-j}\bar \zeta^I$. On the other hand, there is no such  term on the left side of \eqref{transformation of defining equation}, since $\Re\left(h(w,z,\zeta)\right)$ is pluriharmonic, and by the explicit form of $P$, which contains only quadratic terms in $z_1, \dots, z_{n-1}$.
Given that contradiction, we must have $f_{1}^0=\cdots=f_{n-1}^0=0$. Consequently, terms on the right side of \eqref{transformation of defining equation} all have weights greater than $1$ with respect to \eqref{grading0}. Now if $h^0$ or $h^1$ are nonzero then the left side of \eqref{transformation of defining equation} would have terms $\Re(h^0)$ or $\Re(h^1)$, which are of weights $0$ and $1$, contradicting \eqref{transformation of defining equation}. So indeed \eqref{low degree transformation terms} holds.

Collecting terms of weight $2$ in the equation $\Re(w^*)=F^*(z^*,\zeta^*)$ (equivalently, \eqref{transformation of defining equation}), the perturbations $f$ and $f^*$ make no contribution to these terms, and we obtain 
\begin{align}\label{CMO weight 2}
\Re\big(w+h^2(P+iv,z,\zeta)\big)=P(z_1+f_{1}^1,\ldots, z_{n-1}+f_{n-1}^1,\zeta+g_0),
\end{align}
showing that the transformation
\[
(w,z,\zeta)\mapsto (w+h^2,z_1+f_{1}^1,\ldots, z_{n-1}+f_{n-1}^1,\zeta+g^0)
\]
is a symmetry of the model hypersurface $M_0=\{u=P\}$. Recall that the classification of such symmetries is given by Theorem \ref{modelalgebra} and Remark \ref{modelalgebra remark}. By solving the ODE system corresponding to a general vector field in model's real $3$-dimensional isotropy subalgebra,  it follows that there exist real numbers $a_1,a_2,a_3$ such that
\begin{align}\label{transformation solution a}
h^2=(e^{2a_1+na_2}-1)w+ia_3e^{a_1+a_2}z_1^2,
\quad
g^0=(e^{a_2}-1)\zeta,
\end{align}
and
\begin{align}\label{transformation solution b}
f_{j}^1=(e^{a_1+ja_2}-1)z_j+\delta_{j,n-1}ia_3z_1.
\end{align}

With these formulas in hand we can normalize the pair $M_f$ and $M_{f^*}$ by replacing $M_{f^*}$ with an equivalent hypersurface such that a new transformation $\Phi$ between this new pair has the simplified form of \eqref{transformation of hypersurfaces} with
\begin{align}\label{simplified transformation} 
\quad w^*=w+h^3+h^4+\cdots,\,\,
z_j^*=z_j+f_{j}^2+f_{j}^3+\cdots,\,\,\mbox{and}\,\,
\zeta^*=\zeta+g^1+g^2+\cdots.
\end{align}
To achieve \eqref{simplified transformation}, consider the transformation  $\Psi:\mathbb{C}^N\to\mathbb{C}^N$ given by 
\[
(\widehat{w},\widehat{z},\widehat{\zeta})=\Psi(w,z,\zeta):=\left(w+\widehat{h}^2,z_1+\widehat{f}_{1}^1,\ldots, z_{n-1}+\widehat{f}_{n-1}^1, \zeta+\widehat{g}^00\right)
\]
with $\widehat{h}^2$, $\widehat{f}_{j}^1$, and $\widehat{g}^0$ equal to the values of ${h}^2$, ${f}_{j}^1$, and ${g}^0$ in \eqref{transformation solution a} and \eqref{transformation solution b} but with $a_1,a_2,a_3$ replaced by $-a_1,-a_2,-a_3e^{-2a_1-na_2}$ respectively, which corresponds to applying the inverse transformation to the above weight zero symmetry.
Accordingly
\begin{align}
 \Re(\widehat{w})-P\left(\widehat{z},\widehat{\zeta}\right)-f^*\left(\Re(\widehat{z_1})\right)&=e^{-2a_1-na_2}\big(u-P(z,\zeta)\big)-f^*\big(e^{-a_1-a_2}x_1\big),
\end{align}
or, equivalently,
\begin{align}\label{rescaling transformation}
\Psi(M_{f^*(x)})=M_{c_1f^*(c_2x)},
\end{align}
for constants $c_1=e^{2a_1+na_2}$ and $c_2=e^{-a_1-a_2}$, so $M_{f^*(x)}$ and $M_{c_1f^*(c_2x)}$ are equivalent. Let us replace $f^*$ with $c_1f^*(c_2x)$, and replace $\Phi$ with the transformation $\Psi\circ \Phi$ bringing $M_f$ to the new $M_{f^*}$. With these replacements, the new $\Phi$ indeed has the simplified form in \eqref{simplified transformation}, so we can assume without loss of generality that \eqref{simplified transformation} holds.

Now, to produce a contradiction, let us suppose there exists a smallest integer $\mu>0$ such that one of the functions $h^{\mu+2},f_1^{\mu+1},\ldots, f_{n-1}^{\mu+1},g^{\mu}$ is nonzero.
Similar to \eqref{CMO weight 2}, it is easy to see that collecting weight $\mu+2$ terms and rearranging \eqref{transformation of defining equation} yields
\begin{align}\label{CMO weight mu+2}
\Re\left(g^{\mu}P_\zeta -h^{\mu+2}(P+iv,z,\zeta)+\sum_{j=1}^{n-1}f_j^{\mu+1}P_{z_j}\right)=cx_1^{\mu+2}
\end{align}
for some constant $c\in\mathbb{R}$.  
The right side in \eqref{CMO weight mu+2} comes from noticing that our choice of $\mu$ (ensuring $f_{1}^1=\cdots= f_{1}^{\mu}=0$) implies that the $\mu+2$ weighted term in $f\left(\Re(z_1)\right)-f^*\left(\Re(z_1^*)\right)$ depends only on $x_1$ (here one also uses the assumption that $f$ and $f^*$ are normalized as prescribed by Lemma \ref{normalizing quartic terms}).

In terms of the vector field
\[
Y:=h^{\mu+2}\dw+g^{\mu}\dzz+\sum_{j=1}^{n-1}f_{j}^{\mu+1}\dz{j},
\]
\eqref{CMO weight mu+2} implies
\begin{align}\label{CMO weight mu+2 alt}
\Re(Y)(P(z,\zeta)-u)=cx_1^{\mu+2}
\quad\quad\forall\, (w,z,\zeta)\in M_0.
\end{align}
Consider the decomposition $Y=Y^{\prime}+Y^{\prime\prime}$ given by
\[
Y^{\prime}=h^{\mu+2}\dw+f_{n-1}^{\mu+1}\dz{n-1}.
\]
It follows from the explicit form of $P$ that  $\Re(Y^{\prime\prime})(P(z,\zeta)-u)$ does not contain terms that depend only on $x_1$.

Note that no terms in $P(z,\zeta)-u$ depend only on $x_1$, so no terms in $$\Re(Y)(P(z,\zeta)-u)$$
that depend on $w$ will yield terms depending only on $x_1$ after making the substitution $u=P(z,\zeta)$. Hence, setting $z_2 = z_3 = \dots = z_{n-1} = \zeta = v =0$, we obtain 
\begin{align}\label{CMO weight mu+2 b}
ax_1\Re(z_1^{\mu+1})-b\Re(z_1^{\mu+2})=\Re(Y^{\prime})(P(z,\zeta)-u)=cx_1^{\mu+2}
\end{align}
from \eqref{CMO weight mu+2 alt} because all terms in 
\[
x_1\Re(f_{n-1}^{\mu+1})-\Re(h^{\mu+2})=\Re(Y^{\prime})(P(z,\zeta)-u)
\]
(both before and after substituting $u=P(z,\zeta)$) depend on at least one of the variables $z_2,\ldots,z_{n-1},\zeta,w$ except for terms of the form on the left side of \eqref{CMO weight mu+2 b}.
Applying the univariate Laplacian $\Delta$ (in the $z_1$ variable) twice to both sides of \eqref{CMO weight mu+2 b}, we get
\[
0=\Delta^2\left(ax_1\Re(z_1^{\mu+1})-b\Re(z_1^{\mu+2})\right)=\Delta^2(cx_1^{\mu+2})=(\mu+2)(\mu+1)\mu (\mu - 1) cx_1^{\mu-2}
\]
because $\Re(z_1^{\mu+1})$ and $\Re(z_1^{\mu+2})$ are harmonic. 
Therefore the coefficient $c$ (and thus the right side of \eqref{CMO weight mu+2 alt}) is zero, and hence
\begin{align}\label{Y is a symmetry}
Y\in\mathfrak{hol}(M_0,0).
\end{align}

Since $n>4$, \eqref{Y is a symmetry} contradicts Theorem \ref{modelalgebra}. Hence there is no smallest integer $\mu>0$ such that one of the functions $h^{\mu+2},f_{1}^{\mu+1},\ldots, f_{n-1}^{\mu+1},g^{\mu}$ is nonzero.
Therefore, after applying the normalizations of Lemma \ref{normalizing quartic terms} to $f$ and $f^*$ and furthermore normalizing $M_{f^*}$ by replacing it with $\Psi(M_{f^*})$ as in \eqref{rescaling transformation} such that \eqref{simplified transformation} holds, we get $\Phi=\mathrm{Id}$. Since $c_1$ and $c_2$ in \eqref{rescaling transformation} are arbitrary, this shows that after applying the normalizations of Lemma \ref{normalizing quartic terms}, $M_{f}$ and $M_{f^*}$ are equivalent at $0$ if and only if $f^*(x)=c_1f(c_2x)$ for some $c_1,c_2$. Lastly, considering two other functions $\hat{f}$ and $\hat{f}^*$ satisfying
\[
\hat{f}^{(4)}={f}^{(4)}
\quad\mbox{ and }\quad
(\hat{f}^*)^{(4)}=({f}^*)^{(4)},
\]
$(\hat{f}^*)^{(4)}(x)=\hat{c_1}\hat{f}^{(4)}(\hat{c_2}x)$ for some constants $\hat{c_1},\hat{c_2}$ if and only if $f^*(x)=c_1f(c_2x)$  for some other constants $c_1,c_2$, which completes the proof in light of Lemma \ref{normalizing quartic terms}.
\end{proof}

\begin{remark}\label{abelian}
An alternative proof of the above classification theorem could be obtained by using the classification of {\rm maximal Abelian subalgebras} in $\hol(M,0)$ (for more on this tool, see the work of Fels--Kaup \cite{FK11}). In view of Theorem \ref{thealgebras}, having the algebra's multiplication table in hand, such subalgebras are easy to classify: it is the shifts subalgebra spanned by $Y_{-1},...,Y_{-n},Y_{-1}'$ and all of its modifications obtained by replacing $Y_{-1}$ with $aY_{-1}+bX_{1-n}+cV_{2-n}$. The latter modifications can be always annihilated by applying to the subalgebra the flow of $V_{2-n}$ and that of a vector field \eqref{X1} above (in the latter case we also modify $M$ but keep the structure \eqref{main formula}). This reduces the mapping problem for hypersurfaces \eqref{main formula} to that subject to mappings preserving the shift subalgebra, which are all obviously real affine mappings. Observing the need to preserve the commutants of $\hol(M,0)$ of any order, we conclude that any (normalized) map is a real scaling of coordinates. From this the assertion of  Theorem \ref{Holomorphic classification theorem} will follow.  We, however, prefer to stay with the proof of Theorem \ref{Holomorphic classification theorem} given above, as it appears to be shorter and contains some new insight. 
\end{remark}

 \newcommand{\noop}[1]{}

\end{document}